\documentclass[12pt]{amsproc}
\usepackage{amsmath}
\usepackage{amssymb, graphicx}
\usepackage{amsthm}
\usepackage{bm}
\usepackage{verbatim}
\usepackage{setspace}
\usepackage[margin=3.0cm]{geometry}
\usepackage[english]{babel}
\usepackage{tabularx}
\usepackage[all]{xy}
\usepackage{anyfontsize}
\usepackage{color}
\usepackage{enumitem}
\usepackage{hyperref}
\hypersetup{
    colorlinks=false, 
    linktoc=all,     
    linkcolor=blue,
        citecolor=red,
    filecolor=black,
    urlcolor=green
}
\usepackage[all]{hypcap}
\usepackage{layouts}
\usepackage{multibib}
\newcites{sec}{List Of Publications}
\theoremstyle{plain}
\newtheorem{theorem}{Theorem}

\newtheorem{lemma}{Lemma}

\newtheorem{obs}{Observation}

\theoremstyle{definition}
\newtheorem{defn}{Definition}
\newtheorem{remark}{Remark}
\newtheorem{example}{Example}

\newcommand{\mathsym}[1]{{}}

\newcommand{\ga}{\alpha}
\newcommand{\gb}{\beta}
\newcommand{\gga}{\gamma}
\newcommand{\gd}{\delta}
\newcommand{\gep}{\epsilon}

\newcommand{\gl}{\lambda}
\newcommand{\gm}{\mu}

\newcommand{\gp}{\pi}

\newcommand{\gr}{\rho}

\newcommand{\Gl}{\Lambda}

\newcommand{\grpp}{R_{\gl}}
\newcommand{\grppp}{R_{\gl^{'}}}
\newcommand{\grpppp}{R_{\gl^{''}}}

\newcommand{\autgp}{G_{\gl}}

\newcommand{\id}[1]{\mathrm{id}_{#1}}

\newcommand{\subs}{\subset}
\newcommand{\sups}{\supset}

\newcommand{\nin}{\notin}

\newcommand{\ti}{\tilde}

\newcommand{\ugl}{(\gl = \gl_{1}^{\gr _1}>\gl_{2}^{\gr _2}>\gl_{3}^{\gr _3}>\ldots>\gl_{k}^{\gr _k})}
\newcommand{\mattwo}[4]{
\begin{pmatrix}
  #1 & #2\\ #3 & #4
\end{pmatrix}
}

\newcommand{\mbb}{\mathbb}
\newcommand{\mrm}{\mathrm}

\newcommand{\mcl}{\mathcal}

\newcommand{\ul}{\underline}

\newcommand{\us}{\underset}
\newcommand{\os}{\overset}
\newcommand{\lla}{\longleftarrow}
\newcommand{\lra}{\longrightarrow}
\newcommand{\llra}{\longleftrightarrow}

\newcommand{\A}{R}
\newcommand{\N}{\mbb N}
\newcommand{\Z}{\mbb Z}

\newcommand{\AAA}[1]{ R/\gp^{#1} R}
\newcommand{\ZZ}[1]{\Z/p^{#1}\Z}

\newcommand{\bc}{\bigcup}
\newcommand{\es}{\emptyset}

\newcommand{\fo}[3]{
\begingroup
{\fontsize{#1}{#2}\selectfont {#3}}
\endgroup
}

\let\hom\relax
\DeclareMathOperator{\hom}{Hom}
\newcommand{\Hom}{\hom}
\let\End\relax
\DeclareMathOperator{\End}{End}
\newcommand{\Endo}{\End}

\let\oldthebibliography\thebibliography
\renewcommand\thebibliography[1]{
  \oldthebibliography{#1}
 \setlength{\parskip}{0pt}
 \setlength{\itemsep}{0pt plus 0.1ex}
}
\title[Permutation Representations]{Permutation Representations of the Orbits of the Automorphism Group of a Finite Module over Discrete Valuation Ring}
\author{\fo{12}{12}{C.P. ANIL KUMAR}}
\address{The Institute of Mathematical Sciences, Chennai.}
\keywords{Finite Abelian Groups, Finite Group Actions, Automorphism Orbits, Modules over Discrete Valuation Rings, Endomorphism Algebras, Permutation Representations, Multiplicity Free, Distributive Lattice, Congruence Residue Systems}
\subjclass[2010]{20K01,20K30,05E15,20C05,20C33,16S50,06B15,11A07}
\begin{document}
\maketitle
\begin{abstract}
Consider a discrete valuation ring $\A$ whose residue field is finite of cardinality at least $3$.
For a finite torsion module, we consider transitive subsets $O$ under the action of the
automorphism group of the module. We prove that the associated permutation representation on the
complex vector space $C[O]$ is multiplicity free. This is achieved by obtaining a complete
description of the transitive subsets of $O\times O$ under the diagonal action of the automorphism group.
\end{abstract}
\section{Introduction}
Let $\A$ be a discrete valuation ring whose maximal ideal is generated by
a uniformizing element $\pi$, and which has a finite residue field $\mbb{F}_q$. Let $\Gl$
denote the set of all sequences of symbols of the form
\begin{equation*}
(\gl_1^{\gr_1},\gl_2^{\gr_2},\dotsc,\gl_k^{\gr_k}),
\end{equation*}
where $\gl_1>\gl_2>\dotsc>\gl_k$ is a strictly decreasing sequence of positive
integers and $\gr_1,\gr_2,\dotsc,\gr_k$ are positive integers. We allow the case where
$k=0$, corresponding to the empty sequence, which we denote by $\es$.
Every finite $\A$-module $\grpp$ is, up to isomorphism, of the form
\begin{equation}
\label{eq:module-form}
  \grpp = (\AAA {\gl_1})^{\oplus \gr_1}\oplus(\AAA {\gl_2})^{\oplus \gr_2}\oplus\dotsc \oplus
  (\AAA {\gl_k})^{\oplus \gr_k}
\end{equation}
for a unique $\gl \in \Gl$. Let $\autgp$ denote the automorphism
group of $\grpp$.

For a $\gl \in\Gl$ the group $\autgp$ acts on
$\grpp^n$ by the diagonal action
\begin{equation*}
  g\cdot (x_1,\dotsc,x_n) = (g(x_1),\dotsc,g(x_n)) \text{ for $x_i\in \grpp$ and $g\in \autgp$.}
\end{equation*}

For $n=1$, this is just the action on $\grpp$ of its automorphism group $\autgp$.
A description of the orbits for this group action has been available for more than a
hundred years (see Miller~\cite{1905}, Birkhoff~\cite{GarrettBirkhoff01011935},
and Dutta-Prasad~\cite{MR2793603}).

For $n=2$ the group $\autgp \subs \autgp \times \autgp$ is the
diagonal subgroup which acts on $\grpp \times \grpp$. For each orbit $\mcl{O}$ there exist
certain representatives $e(\mcl{O})$ called canonical forms.
The transitive subsets of $R_{\lambda}$ under the stabilizer of $e(O)$ have been described in
~\cite{AnilAmri}, Theorem $27$.

We obtain a description of transitive subsets for the diagonal
action of $\autgp$ on $\mcl{O} \times \mcl{O}$ in
Theorems~\ref{theorem:CanonicalSimilarPair}~\ref{theorem:GeneralOrbitPairMax}~\ref{theorem:GeneralOrbitPairNonMax}~\ref{theorem:GeneralOrbitPairConverse}.
We use this description to show that the permutation representation $\mbb{C}[\mcl{O}]$ of
$\autgp$ on an orbit $\mcl{O}$ is multiplicity free.

Let $R$ be any DVR with residue field of cardinality at least three. Let $\grpp$ be a finite $R-$module corresponding to partition $\ugl$.
Then we prove the following main result on multiplicity one:
\begin{theorem}
\label{theorem:MultFree1}
With the above notations, for any orbit $\mcl{O}$ of $\autgp$, the permutation representation
$\mbb{C}[\mcl{O}]$ of $\autgp$ is multiplicity free.
\end{theorem}

A more precise statement is given in Theorem~\ref{theorem:MultFree}
after introducing a combinatorial description of the orbits.

\section{Sketch of the Proof of Theorem}
Let $G$ be a finite group. For a finite $G-$set $S$ consider the permutation representation
$\mbb{C}[S]$. The endomorphism algebra $End(\mbb{C}[S])$ can be
identified as a $G-$representation algebra with the space $\mbb{C}[S
\times S]$. Under
this correspondence the endomorphism sub-algebra
$End_{G}(\mbb{C}[S]) \subs End(\mbb{C}[S])$ corresponding to the trivial
sub-representation can be identified with the subspace $\mbb{C}_G[S
\times S] \subs \mbb{C}[S \times S]$ of complex valued functions on
$S \times S$ which are constant on the transitive subsets of $S
\times S$. Hence there exists a
canonical basis $I_{\mcl{O}}$ for the endomorphism subalgebra
$End_{G}(\mbb{C}[S])$ consisting of indicator functions of these
transitive subsets $\mcl{O}$ of $S \times S$. The algebra structure
on the space $\mbb{C}[S \times S]$ arises via convolution of
functions making it a convolution algebra. To prove that the
permutation representation $\mbb{C}[S]$ is multiplicity free is
equivalent to proving that the convolution algebra $\mbb{C}_G[S
\times S]$ is commutative.(See Theorems~\ref{theorem:CanonicalSimilarPair}~\ref{theorem:GeneralOrbitPairMax}~\ref{theorem:GeneralOrbitPairNonMax}~\ref{theorem:GeneralOrbitPairConverse}). It is enough to prove that the
convolution is commutative for the basis elements. So we reduce the
commutativity question
\begin{equation*}
I_{\mcl{O}_1}*I_{\mcl{O}_2}=I_{\mcl{O}_2}*I_{\mcl{O}_1}
\end{equation*}
of convolution of basis functions $I_{\mcl{O}}$ to a counting question of cardinality of
certain sets. Then the following two sets have the same cardinality. For every
$(x,y)\in S \times S$
\begin{equation*}
\mid\{z \in S\mid (x,z) \in \mcl{O}_1 \text{ and } (z,y) \in \mcl{O}_2\}\mid
= \mid\{z \in S\mid (x,z) \in \mcl{O}_2 \text{ and } (z,y) \in
\mcl{O}_1\}\mid.
\end{equation*}

Here we introduce a definition.
\begin{defn}
\label{def:Combinatorial}
We say a set $S \subs R_{\gl}$ is a valuative combinatorial set if it can be described by valuations
which have combinatorial significance. In other words $S \in \mcl{L}$ where $\mcl{L}$ is a collection
of subsets of $R_{\gl}$ with a suitable lattice structure (additively closed for supremum).
We also say in that case $S$ has a valuative combinatorial description.
See the appendix section~\ref{sec:Appendix} for the motivation to this definition.
\end{defn}
\begin{example}
The following sets with valuative combinatorial descriptions arise naturally in our study;
the right hand sides introduce the notation for each of them that we will use throughout the paper.
\begin{itemize}
\item $(\AAA {n}) = \A_n$.
\item $\gp^j(\AAA {n})^k=\gp^j\A_n^k$.
\item $(\AAA {n})-\gp^{1}(\AAA {n})=\{u\mid u \text{ is a unit in }\AAA
{n}\}$ by $\A_n^{*}$.
\item $\gp^j(\AAA {n})-\gp^{j+1}(\AAA {n})
=\{\gp^ju\mid u \text{ is a unit in }\\
\AAA {n}\}=\gp^j\{u\mid u \text{ is a unit in }\AAA {n}\}$ by $\gp^j\A_n^{*}$.
\item The set
$(\AAA{n})^k-\gp^{1}(\AAA{n})^{k}$ by $\A_n^{k*}$
\item The set
$\gp^j(\AAA{n})^k-\gp^{j+1}(\AAA{n})^{k}=\gp^j\big((\AAA{n})^k-\gp^{1}(\AAA{n})^{k}\big)$
by $\gp^j\A_n^{k*}$.
\end{itemize}
\end{example}

We will crucially use the following result which was proved in~\cite{AnilAmri}, Theorem $27$.

Let $\ugl$ be a partition and $\A_{\gl}$ the corresponding finite $R-$module.
For any orbit $\mcl{O} \subs \A_{\gl}$, and an element $e$ in $\mcl{O}$, consider its stabilizer $(\autgp)_e$.
Then, we have:

\begin{theorem}
\label{theorem:sub-orbit1}
For any $l \in \grpp$ the $(\autgp)_e-$orbit of $l$ is isomorphic (automorphisms and translations are allowed) to a set which has a valuative
combinatorial description for a suitable partition decomposition of $\gl$.
\end{theorem}

See~\cite{AnilAmri}, Theorem $27$ for a more precise version using
combinatorial description of the orbits.

For a partition $\gl$ and an orbit $\mcl{O}$ in $\grpp$, consider the stabilizer
$(\autgp)_e$ of an element $e$ of $\mcl{O}$. Suppose $\mcl{O}_1 \subs \mcl{O}$ is an orbit under this stabilizer
subgroup. We know that there is a corresponding orbit $\ti{\mcl{O}} \in \mcl{O}\times \mcl{O}$ for the diagonal action of $G_{\lambda}$.
We note that $\ti{\mcl{O}}$ can also be obtained via canonical form $e(\mcl{O})$ instead of an arbitrary element
$e \in \mcl{O}$. Then, we prove the theorem:

\begin{theorem}
\label{theorem:sub-orbit2}
With the above notations, $\ti{\mcl{O}}$ is a set which has a
combinatorial description with some additional properties such as being valuative or having linearly
independent conditions and described componentwise with respect to
the isotypic parts of $\grpp$.
\end{theorem}

This theorem is restated in Theorems~\ref{theorem:CanonicalSimilarPair}~\ref{theorem:GeneralOrbitPairMax}
~\ref{theorem:GeneralOrbitPairNonMax}~\ref{theorem:GeneralOrbitPairConverse} more precisely.
This componentwise description with respect to the
isotypic parts of $\grpp$ is usually not present in description of
stabilizer orbits in Theorem~\ref{theorem:sub-orbit1} because of partition decomposition of $\gl$.
The componentwise description in Theorem~\ref{theorem:sub-orbit2} is
useful to ease the proof of commutativity as described in the next
subsection.

\subsubsection{Strategy to prove Commutativity of Convolution by Counting}
~\\
Once we have the description of a transitive set in an orbit of pairs $\mcl{O}
\times \mcl{O}$ for the diagonal action of $\autgp$ we prove
commutativity of the convolution operation in the endomorphism
algebra $\mbb{C}_{\autgp}[\mcl{O} \times \mcl{O}]$.

We wish to show that for two transitive sets $\mcl{O}_1,\mcl{O}_2$ of $\mcl{O} \times \mcl{O}$ that
\begin{equation*}
\mid\{z \in S\mid (x,z) \in \mcl{O}_1 \text{ and } (z,y) \in \mcl{O}_2\}\mid
= \mid\{z \in S\mid (x,z) \in \mcl{O}_2 \text{ and } (z,y) \in
\mcl{O}_1\}\mid.
\end{equation*}

For the purpose of counting we use a general section of the residue
field into the quotients of the discrete valuation ring which will
just enable us to count solutions to the congruence equations modulo
a power of the uniformizer. We do not require additional properties
of the section. This counting is done isotypic componentwise.

The Lemmas~\ref{lemma:LemmaCommutativity},~\ref{lemma:countingvectors},~\ref{lemma:LinearIndependences}
prove the counting results needed for commutativity.
Then Theorem~\ref{theorem:Commutativity} establishes commutativity.

The following section~\ref{sec:preliminaries} describes some
preliminary results essential to state the main results of
this paper.
\section{\bf{Preliminaries}}
\label{sec:preliminaries}
\subsection{\bf{Fundamental Poset and Characteristic Submodules}}
\label{sec:orbits-elements}

An elaborate description of the results in this section is given in
Dutta and Prasad~\cite{MR2793603}.

For any module $\grpp$ of the form given in
equation~(\ref{eq:module-form}), the $\autgp$\nobreakdash-orbits in
$\grpp$ are in bijective correspondence with a certain class of
ideals in a poset $\mcl{P}$, which we call the fundamental poset. As
a set,
\begin{equation*}
  \mcl{P} = \{(v,k)\mid k\in \mbb{N},\; 0\leq v<k\}.
\end{equation*}
The partial order on $\mcl{P}$ is defined by setting
\begin{equation*}
  (v,l) \leq (v',l') \text{ if and only if } v \geq v' \text{ and } l-v \leq l'-v'.
\end{equation*}

Let $J(\mcl{P})$ denote the lattice of order ideals in $\mcl{P}$. A
typical element of $\grpp$ from equation~(\ref{eq:module-form}) is a
vector of the form
\begin{equation*}
  e = (e_{\gl_i,t_i}),
\end{equation*}
where $i$ runs over the set $\{1,\dotsc,k\}$, and for each $i$,
$t_i$ runs over the set $\{1,\dotsc,\gr_i\}$. Let
$val_{\gp}(e_{\gl_i,t_i})$ be the unique integer $j$ such that
$e_{\gl_i,t_i} = u\gp^j$ for some unit $u$ in $\AAA {\gl_i}$. In
particular, $val_{\gp}(0) = \infty$. To $e\in \grpp$ associate the
order ideal $I(e)\subset \mcl{P}$ generated by the elements
\begin{equation*}
  (\us{t_i \in \{1,\dotsc,\gr_i\}}{\min} \ val_{\gp}(e_{\gl_i,t_i}), \gl_i)
\end{equation*}
for all $t_i$ and for all $i$ such that the coordinate
$e_{\gl_i,t_i}\neq 0$ in $\AAA {\gl_i}$. The order ideal
$I(0)$ is the empty ideal.

Consider for example, in the finite abelian $p$-group (or $\Z_p$-module):
\begin{equation}
  \label{eq:fapg-eg}
  \grpp = \ZZ {5}\oplus\ZZ {4}\oplus\ZZ {4}\oplus\ZZ {2}\oplus\ZZ {1}.
\end{equation}
the order ideal $I(e)$ of $e = (0,up,p^2,vp,1)$, where $u$ and $v$
are coprime to $p$. It is the ideal generated by
$\{(1,4),(1,2),(0,1)\}$ in the fundamental poset.

A key observation of Dutta and Prasad~\cite{MR2793603} is the following theorem:
\begin{theorem}
  \label{theorem:ideals-homomorphisms}
  Let $\grpp$ and $R_{\gm}$ be two finite $\A$-modules.
  An element $f \in R_{\gm}$ is a homomorphic image of $e \in \grpp$
  (in other words, there exists a homomorphism $\phi:\grpp\to R_{\gm}$
  such that $\phi(e)=f$) if and only if $I(e)\sups I(f)$.
\end{theorem}
For any two elements $e,f \in \grpp$ the existence of two
homomorphisms, one which takes $e$ to $f$ and the other which takes
$f$ to $e$ makes $e$ and $f$ automorphic to each other i.e. they lie
in the same $\autgp$\nobreakdash-orbit. The following establishes this partial order on the $\autgp$\nobreakdash-orbits of
elements of an abelian group.
\begin{theorem}
  \label{theorem:ideals-and-orbits}
  If $I(e)=I(e')$ for any $e,e'\in \grpp$, then $e$ and $e'$ lie in the same $\autgp$\nobreakdash-orbit.
\end{theorem}
For each order ideal $I\subset \mcl{P}$, let $\max(I)$ denote its
set of maximal elements. Let $J(\mcl{P})_{\gl}$ denote the
sublattice of $J(\mcl{P})$ consisting of ideals such that $\max(I)$
is contained in the set
\begin{equation*}
  \mcl{P}_\gl = \{(v,l) \in \mcl{P} \mid l=\gl_i \text{ for some } 1\leq i \leq k\}.
\end{equation*}
Then the $\autgp$\nobreakdash-orbits in $\grpp$ are in bijective
correspondence with the order ideals in $J(\mcl{P})_{\gl}$. For each
order ideal $I\subs \mcl{P}$, we use the notation $\grpp^{I^*}$ for
the orbit corresponding to $I$ and
\begin{equation*}
  \grpp^{I^*} = \{e\in \grpp\mid I(e)=I\}
\end{equation*}
A convenient way to think about ideals in $\mcl{P}$ is in terms of what
we call their boundaries: for each positive integer $k$ define the boundary valuation of $I$ at
$l$ to be
\begin{equation}
\label{eq:Boundary}
  \partial_l I = \min\{v\mid (v,l)\in I\}.
\end{equation}
We denote the sequence $\{\partial_l I\}$ of boundary valuations by $\partial I$ and call it the boundary of $I$.
This is indeed the boundary of the region corresponding to the ideal $I$ in the fundamental poset.

The ideal $I$ is completely determined by $\max(I)$: in fact taking
$I$ to $\max(I)$ gives a bijection from the lattice
$J(\mcl{P})_{\gl}$ to the set of antichains in $\mcl{P}_\gl$.

\begin{theorem}
  \label{theorem:description-of-orbits}
  The orbits $\grpp^{I^*}$ consists of elements $e=(e_{\gl_i,t_i})$
  such that $val_{\gp}(e_{\gl_i,t_i})$\\ $\geq \partial_{\gl_i}I$ for all $\gl_i$ and $t_i$, and such that $val_{\gp}(e_{\gl_i,t_i}) = \partial_{\gl_i}I$ for at least one $t_i$ if $(\partial_{\gl_i} I, \gl_i)\in \max(I)$.
\end{theorem}
In other words, the elements of $\grpp^{I^*}$ are those elements all of whose
coordinates have valuations not less than the corresponding boundary valuation,
and at least one coordinate corresponding to each maximal element of $I$ has
valuation exactly equal to the corresponding boundary valuation.

For each $I\in J(\mcl{P})_{\gl}$, with
\begin{equation*}
  \max(I) = \{(v_1,l_1),\dotsc,(v_s,l_s)\}
\end{equation*}
define an element $e(I)$ of $\grpp$ whose coordinates are given by
\begin{equation*}
  e_{\gl_i,t_i} =
  \begin{cases}
    \gp^{v_j} & \text{ if } \gl_i = l_j \text{ and } t_j = 1\\
    0 & \text{ otherwise}.
  \end{cases}
\end{equation*}
In other words, for each element $(v_j,l_j)$ of $\max I$, pick $\gl_i$ such that $\gl_i=l_j$.
In the summand $(\AAA {\gl_i})^{\oplus \gr_i}$, set the first coordinate of $e(I)$ to $\gp^{v_j}$,
and the remaining coordinates to $0$.

\begin{theorem}
  \label{theorem:orbit-par}
Let $\grpp$ be a finite $\A$-module of the type given in
equation~(\ref{eq:module-form}). The functions for any transitive
orbit $\mcl{O} \mapsto I(e)$ for any $e \in \mcl{O}$ and $I \mapsto
\grpp^{I^*}$ the orbit of $e(I)$ are mutually inverse bijections
between the set of $\autgp$\nobreakdash-orbits in $\grpp$ and the
set of order ideals in $J(\mcl{P})_{\gl}$.
\end{theorem}
We shall say that an element of $\grpp$ is a canonical form if it is
equal to $e(I)$ for some order ideal $I\in J(\mcl{P})_{\gl} =
J(\mcl{P}_{\gl})$.

The set of endomorphic images of elements in this orbit is a $\autgp$-invariant submodule of
$\grpp$ which we denote by $\grpp^I$. We have
\begin{equation}
\label{eq:Char-Submodule}
      \grpp^I = \bigsqcup_{\{J\in J(\mcl{P})_{\gl}\mid J\subs I\}}\grpp^{J^*}.
\end{equation}
This submodule is a characteristic submodule as it is a union of $\autgp$ invariant sets
(a submodule of $\grpp$ is said to be characteristic if it is a
$\autgp$ invariant submodule of $\grpp$).

The description of $\grpp^I$ in terms of valuations of coordinates and boundary valuations is very simple:
\begin{equation}
  \label{eq:M_I}
  \grpp^I = \{ e=(e_{\gl_i,t_i})\mid val_{\gp}(e_{\gl_i,t_i})\geq \partial_{\gl_i} I\}.
\end{equation}

The map $I\mapsto \grpp^I$ is not injective on
$J(\mcl{P})$. It becomes injective when restricted to
$J(\mcl{P})_{\gl}$. For example, if $J$ is the order ideal in
$\mcl{P}$ generated by $(2,6)$, $(1,4)$ and $(0,1)$, then the ideal
$J$ is strictly larger than the ideal $I$, but when $\grpp$ is as
given in equation~(\ref{eq:fapg-eg}), $\grpp^I=\grpp^J$.

The $\autgp$\nobreakdash-orbits in $\grpp$ are parametrized by the
finite distributive lattice $J(\mcl{P})_{\gl}$. Moreover, each order
ideal $I\in J(\mcl{P})_{\gl}$ gives rise to a $\autgp$-invariant
submodule $\grpp^I$ of $\grpp$. The lattice structure of
$J(\mcl{P})_{\gl}$ gets reflected in the poset structure of the
submodules $\grpp^I$ when they are partially ordered by inclusion:

\begin{theorem}
  \label{theorem:Ideal-CharSubModules}
  Let $\grpp$ be a finite $\A$ module as given in equation~(\ref{eq:module-form}).
  The function $I \mapsto \grpp^I$, with $\grpp^I$ as in equation~(\ref{eq:Char-Submodule}),
  is a lattice isomorphism between the the set of order ideals in $J(\mcl{P})_{\gl}$ and the
  set of characteristic submodules of $\grpp$ of the form $\grpp^I$.
\end{theorem}

In other words, for ideals $I,J\in J(\mcl{P})_{\gl}$,
\begin{equation*}
  \grpp^{I\cup J} = \grpp^I + \grpp^J \text{ and } \grpp^{I \cap J} = \grpp^I \cap \grpp^J.
\end{equation*}
In fact, when $\mbb{F}_q \cong \AAA {1}$ with $|\mbb{F}_q|$ having
at least three elements, every $\autgp$-invariant submodule is of
the form $\grpp^{I}$, therefore $J(\mcl{P})_{\gl}$ is isomorphic to
the lattice of $\autgp$-invariant submodules (Kerby and
Rode~\cite{MR2782608}).

This is not true for $q=2$. Consider the abelian group $\Z/2^3\Z
\oplus \Z/2\Z$ and the subgroup $H = \{(0,0),(2,1), 2(2,1) = (4,0),
3(2,1) = (6,1)\}$. This subgroup is characteristic but it does not
correspond to any ideal in $J(\mcl{P})_{\gl}$ where $\gl = (3,1) \in
\Gl$.

\subsection{\bf{Stabilizer of Canonical Forms}}
\label{sec:stab-canon-forms} Results in this subsection are proved
elaborately in~\cite{AnilAmri}. We introduce a decomposition of
$\grpp$ into a direct sum of two $\A$-modules (and this
decomposition depends on the ideal $I$):

\begin{equation}
  \label{eq:decomposition}
  \grpp = \grppp\oplus \grpppp,
\end{equation}
where $\grppp$ consists of those cyclic summands in the
decomposition given in equation~(\ref{eq:module-form}) of $\grpp$
where $e(I)$ has non-zero coordinates, and $\grpppp$ consists of the
remaining cyclic summands. In the example~\ref{eq:fapg-eg} and the
order ideal $I(e)$ of $e = (0,up,p^2,vp,1)$, where $u$ and $v$ are
coprime to $p$, we have
\begin{equation*}
  \grppp = \Z/p^4\Z\oplus \Z/p\Z,\quad \grpppp = Z/p^5\Z\oplus\Z/p^4\Z\oplus\Z/p^2\Z.
\end{equation*}
Let the first projection of $e(I)$ be $e(I)^{'} \in \grppp$.
The reason for introducing this decomposition is that the description of the stabilizer of
$e(I)^{'}$ in the automorphism group of $\grppp$ is quite nice:
\begin{theorem}
\label{lemma:stab}
  The stabilizer of $e(I)^{'} \neq 0$ in $G_{\gl^{'}}$ is
  \begin{equation*}
   G_{\gl^{'}}^I = \{\id{\grppp}+n\mid n\in \mrm{End}_{\A}(\grppp)\text{ satisfies }n(e(I)')=0\}.
  \end{equation*}
\end{theorem}
Theorem~\ref{lemma:stab} follows from the following
Theorem~\ref{lemma:nilpotent}.
\begin{theorem}
  \label{lemma:nilpotent}
  For any $\A$-module of the form
  \begin{equation*}
    R_{\gm} = \AAA {\gm_1}\oplus\dotsb\oplus \AAA {\gm_m},
  \end{equation*}
  with $\gm_1>\dotsb>\gm_m$,
  and $x=(\gp^{v_1},\dotsc,\gp^{v_m})\in R_{\gm}$ such that the set
  \begin{equation*}
    (v_1,\gm_1),\dotsc,(v_m,\gm_m)
  \end{equation*}
  is an antichain in $\mcl{P}$, if $n\in \mrm{End}_{\A}R_{\gm}$ is such that $n(x)=0$, then $n$ is nilpotent.
\end{theorem}

Every endomorphism of $\grpp$ can be written as a matrix $\mattwo xyzw$, where
$x:\grppp\to \grppp$, $y:\grpppp\to\grppp$, $z:\grpppp\to \grppp$ and $w:\grpppp\to
\grpppp$ are homomorphisms.

We are now ready to describe the stabilizer of $e(I)$ in $\autgp$:

\begin{theorem}
  \label{theorem:stabilizer}
Let
\begin{itemize}
\item $N^{\gl^{'}} = \{n \in \mrm{End}_{\A}(\grppp) \mid n(e(I)')=0\}$ is a nilpotent
ideal in $\mrm{End}_{\A}(\grppp)$.
\item $M(\gl^{'},\gl^{''}) = \{z\in \Hom_{\A}(\grppp,\grpppp) \mid z(e(I)')=0\}$.
\end{itemize}
The stabilizer of $e(I)$ in $\autgp$ consists of matrices of the
form
  \begin{equation*}
    \mattwo{\id{\grppp}+n}yzw,
  \end{equation*}
  where $n\in N^{\gl^{'}} \subs \mrm{End}_{\A}(\grppp)$,  $y\in \Hom_{\A}(\grpppp,\grppp)$
  is arbitrary, $z\in M(\gl^{'},\gl^{''}) \subs \Hom_{\A}(\grppp,
  \grpppp)$, and $w\in G_{\gl^{''}}$ is invertible.
\end{theorem}
\subsection{\bf{The Stabilizer $\autgp^I$ Orbit of an Element in $\grpp$}}
\label{sec:stabilizer-orbit-an} Results in this subsection are also
proved elaborately in~\cite{AnilAmri}. Let $\autgp^I$ denote the
stabilizer of $e(I)\in \grpp$. Write each element $m \in \grpp$ as
$m=(m',m'')$ with respect to the decomposition given in the
equation~(\ref{eq:decomposition}) of $\grpp$. Also, for any $m'\in
\grppp$, let $\bar m'$ denote the image of $m'$ in $\grppp/\grppp
e(I)'$.

Theorem~\ref{theorem:stabilizer} allows us to describe the orbits of $m$ under the action of
$\autgp^I$, which is the same as describing the $\autgp$\nobreakdash-orbits
in $\grpp\times\grpp$ whose first component lies in the orbit $\grpp^{I^*}$ of $e(I)$.
\begin{theorem}
  \label{theorem:sub-orbit}
  Given $l$ and $m$ in $\grpp$, $l$ lies in the $\autgp^I$\nobreakdash-orbit of $m$ in $\grpp$
  if and only if the following conditions hold:
  \begin{itemize}
  \item $l'\in m' + \grppp^{I(\bar m')\cup I(m'')}$.
  \item \label{item:second-condition} $l''\in \grpppp^{{I(m'')}^*}+\grpppp^{I(\bar m')}$.
  \end{itemize}
(Here $I(\bar{m}'),I(m'')$ are ideals in $J(\mcl{P})_{\gl}$)
\end{theorem}

\section{\bf{Two Simple Cases}}
In this section we describe two simple cases when the finite torsion
module is cyclic. We begin with a remark.
\begin{remark}
A complex representation $V$ of a finite group $G$ is multiplicity
free if and only if the endomorphism algebra $\Endo_{G}[V]$ is
commutative.
\end{remark}
\begin{theorem}
For $\gl = (n) \in \Gl$, the permutation representation of $\autgp$
on any $\autgp$\nobreakdash-orbit $\grpp^{I^*}$ in $\grpp$ is
multiplicity-free.
\end{theorem}
\begin{proof}
In this case we have that the number of orbits of the group under
the automorphism group is $(n+1)$.
\begin{center}
\begin{tabular}{ | l | l | l |}
\hline Ideal I & Orbit $\grpp^{I^*}$ & $\grpp^I$\\ \hline $\max\ I =
\{(j,n)\} \text{ such that } 0\leq j < n$ & $\gp^j\A_n^*$ &
$\gp^j(\AAA {n})$\\ \hline $\max\ I = \es$ & $\{\ul{0}\}$ & $0$\\
\hline
\end{tabular}
\end{center}
Consider the orbit $\grpp^{I^*}$ with $\max(I) = \{(j,n)\}$. For
each $y \in (\AAA {n-j})^{*}$, let
\begin{equation*}
  (\grpp^{I^{*}})^y = \{(a,ya) | \text{ for all } a \in \grpp^{I^*}\}.
\end{equation*}
The partition of $\grpp^{I^*} \times \grpp^{I^*}$ into transitive
subsets under the diagonal action of $\autgp$ is given by
\begin{equation*}
  \grpp^{I^*} \times \grpp^{I^*} = \bigsqcup_{y\in  (\AAA {n-j})^{*}} (\grpp^{I^*})^y.
\end{equation*}
Let $I_{y}$ denote the indicator function of $(\grpp^{I^*})^y$. Then
we have $I_{y_1} \ast I_{y_2} = I_{y_1y_2} \text{ for all }
y_1,y_2$\\ $\in (\AAA {n-j})^{*}$ which is obvious in this case
(convolution of two lines in the plane corresponding to
$(\grpp^{I^{*}})^{y_1}$ and $(\grpp^{I^{*}})^{y_2}$ passing through
the origin with slopes not in the set $\{0, \infty\}$ correspond to
multiplication of their slopes $y_1,y_2$). So the endomorphism
algebra $\Endo_{\autgp}(\mbb{C}[\grpp^{I^{*}}])$ is commutative. The
permutation representation on the zero orbit is the trivial
representation. Hence the permutation representation on any orbit in
the case $\gl = (n) \in \Gl$ is multiplicity-free.
\end{proof}

\begin{theorem}
\label{theorem:SingleComponentCase} For $\gl = (n^k) \in \Gl$ with
$k>1$ the permutation representation of $\autgp$ on any
$\autgp$\nobreakdash-orbit $\grpp^{I^*}$ in $\grpp$ is
multiplicity-free.
\end{theorem}

\begin{proof}
Here also the number of orbits of the group under the automorphism
group is $(n+1)$.
\begin{center}
\begin{tabular}{ | l | l | l |}
\hline Ideal $I$ & Orbit $\grpp^{I^*}$ & $\grpp^I$\\
\hline $\max\ I = \{(j,n)\}, 0 \leq j < n$ & $\gp^j\A_n^{k*}$ &
$\gp^j(\AAA {n})^k$\\ \hline $\max\ I = \es$ & $\{\ul{0}\}$ & $0$\\
\hline
\end{tabular}
\end{center}

Consider the orbit $\grpp^{I^*}$ with $\max(I) = \{(j,k)\}$. Again
for each $y \in (\AAA {n-j})^{*}$ we have a transitive subset
$(\grpp^{I^*})^y$ of $\grpp^{I^*} \times \grpp^{I^*}$ defined as
$(\grpp^{I^*})^y = \{(a,ya) | \text{ for all } a \in \grpp^{I^*}\}$.
The complement of $\us{y\in(\AAA
{n-j})^{*}}{\bigcup}(\grpp^{I^*})^y$ in $\grpp^{I^*}\times
\grpp^{I^*}$ is also a transitive set, which we denote by $\mcl
O_g$. This is the largest transitive set which arises only when
$k>1$. Thus the decomposition of $(\grpp^{I^*} \times \grpp^{I^*})$
into transitive sets under the diagonal action of $\autgp$ is given
by
\begin{equation*}
(\grpp^{I^*} \times \grpp^{I^*}) = \mcl{O}_{g}\bigsqcup
\Big(\us{y\in(\AAA {n-j})^{*}}{\bigsqcup}(\grpp^{I^*})^y\Big).
\end{equation*}
Here the transitive subsets are parametrized by the set
$(\AAA {n-j})^{*} \cup \{g\}$. Let $I_{y}$ denote the indicator
function of $(\grpp^{I^*})^y$. Then we have $I_{y_1} \ast I_{y_2} =
I_{y_1y_2} \text{ for all } y_1,y_2 \in (\AAA {n-j})^{*}$. Let
$I_{g}$ denote the indicator function of $\mcl{O}_{g}$ then $I_g$
commutes with $I_y$ because the indicator function of the whole set
$\grpp^{I^*} \times \grpp^{I^*}$ commutes with $I_{y}$ for each $y$.
So the endomorphism algebra $\Endo_{\autgp} (\mbb{C}[\grpp^{I^*}])$
is commutative. The permutation representation on the zero orbit is
trivial and 1-dimensional. Hence the permutation representation on
any orbit in the case $\gl = (n^k) \in \Gl$ is multiplicity-free.
\end{proof}

\section{\bf{Transitive Subsets in Similar Orbit of Pairs}}
Here we provide a general description for a transitive subset
$\mcl{O}$ of pairs in similar orbit $\grpp^{I^{*}}$ of pairs
$\grpp^{I^{*}} \times \grpp^{I^{*}}$ for the diagonal action of
$\autgp$ which has a bijective correspondence with an orbit under
the action of the stabilizer $\autgp^I$ on $\grpp^{I^{*}}$ using
Observation~\ref{obs:Bijection}.

First of all we observe that each orbit $\grpp^{I^{*}} \subs \grpp$
intersects any isotypic part $\A_{\gl_i}^{\gr_i}=(\AAA
{\gl_i})^{\gr_i}$ of $\grpp$ in a set of the form
\begin{enumerate}[label=\arabic*$\bm{.}$]
\item $\gp^{\partial_{\gl_i}I}\A_{\gl_i}^{\gr_i*}=\gp^{\partial_{\gl_i}I}
(\AAA {\gl_i})^{\gr_i} - \gp^{(\partial_{\gl_i}I)+1}(\AAA {\gl_i})^{\gr_i}$
if $(\partial_{\gl_i}I,\gl_i) \in \max(I)$.
\item $\gp^{\partial_{\gl_i}I}\A_{\gl_i}^{\gr_i}=\gp^{\partial_{\gl_i}I}
(\AAA {\gl_i})^{\gr_i}$ if $(\partial_{\gl_i}I,\gl_i) \nin \max(I)$.
\end{enumerate}
and moreover $\grpp^{I^{*}}$ is the product of these intersections.
We shall show in Section~\ref{sec:OrbitOfPairs}, in
Theorems~\ref{theorem:CanonicalSimilarPair}~\ref{theorem:GeneralOrbitPairMax}~\ref{theorem:GeneralOrbitPairNonMax}~\ref{theorem:GeneralOrbitPairConverse} that each transitive subset
$\mcl{O}\subs \grpp^{I^{*}} \times \grpp^{I^{*}} \subs \grpp \times
\grpp$ consists a set of ordered pairs which also has a component
structure. The $i^{th}$ component is a subset of
$\gp^{\partial_{\gl_i}I}\A_{\gl_i}^{\gr_i*} \times
\gp^{\partial_{\gl_i}I}\A_{\gl_i}^{\gr_i*}$ if
$(\partial_{\gl_i}I,\gl_i) \in \max(I)$ whereas it is a subset of
$\gp^{\partial_{\gl_i}I}\A_{\gl_i}^{\gr_i} \times
\gp^{\partial_{\gl_i}I}\A_{\gl_i}^{\gr_i}$ if
$(\partial_{\gl_i}I,\gl_i) \nin \max(I)$. The component sets has the
following description given below. Moreover $\mcl{O}$ is the product
of these intersections of ordered pairs.
\\
In case $1$, if $(\partial_{\gl_i}I,\gl_i) \in \max(I)$
\begin{itemize}
\item $\{(a,b) \in \gp^{\partial_{\gl_i}I}\A_{\gl_i}^{\gr_i*} \times
\gp^{\partial_{\gl_i}I}\A_{\gl_i}^{\gr_i*} \mid b-ay \in
\gp^r\A_{\gl_i}^{\gr_i} \text{ for some } r > \partial_{\gl_i}I\\
\text{ and for some } y \in \A^{*}\}$\\
$\textbf{ OR }$\\
\item $\{(a,b) \in \gp^{\partial_{\gl_i}I}\A_{\gl_i}^{\gr_i*} \times
\gp^{\partial_{\gl_i}I}\A_{\gl_i}^{\gr_i*} \mid b-ay \in
\gp^r\A_{\gl_i}^{\gr_i*} \text{ for some } r \geq
\partial_{\gl_i}I\\ \text{and for some } y \in \A^{*}
\text{ and } \gp^{-r}(b-ay) (mod\ \ \gp) \text{ is linearly
independent with }\\ \gp^{-\partial_{\gl_i}I}a (mod\ \ \gp) \text{
in } (\mbb{F}_q)^{\gr_i}\}$
\end{itemize}
In case $2$, if $(\partial_{\gl_i}I,\gl_i) \nin \max(I)$
\begin{itemize}
\item $\{(a,b) \in \gp^{\partial_{\gl_i}I}\A_{\gl_i}^{\gr_i} \times
\gp^{\partial_{\gl_i}I}\A_{\gl_i}^{\gr_i}\}$(The total product set)\\
$\textbf{ OR }$\\
\item $\{(a,b) \in \gp^{\partial_{\gl_i}I}\A_{\gl_i}^{\gr_i} \times
\gp^{\partial_{\gl_i}I}\A_{\gl_i}^{\gr_i} \mid b-ay \in \gp^r\A_{\gl_i}^{\gr_i} \text{
for some } r > \partial_{\gl_i}I\\ \text{ and for some } y \in \A^{*}\}$\\
$\textbf{ OR }$\\
\item $\{(a,b) \in \gp^{\partial_{\gl_i}I}\A_{\gl_i}^{\gr_i} \times
\gp^{\partial_{\gl_i}I}\A_{\gl_i}^{\gr_i} \mid b-ay \in
\gp^r\A_{\gl_i}^{\gr_i*} \text{ for some } r \geq \partial_{\gl_i}I\\
\text{ and for some } y \in \A^{*}\}$
\end{itemize}
Here we note that in each case the values $r,y$ are fixed in the description.
The total product set does not arise in case $1$.
In the description of the sets in second and third parts of case $2$
for some fixed unit $y\in R^{*}$
\begin{itemize}
\item Part(2):The difference $b-ay$ should have at least a fixed higher $\gp-$valuation $r$ than $\partial_{\gl_i}I$
which is a combinatorial set.
\item Part(3):The difference $b-ay$ can have exact $\gp-$valuation $\partial_{\gl_i}I$ but then by definition of valuation of
a vector it will be a multiple of a vector which has a unit in a component along with the exact $\gp-$valuation
less than $\gr_i$ which is a again a combinatorial set. Otherwise it is a zero vector.
\end{itemize}

As an example we have seen in the single component case or cyclic
module case in Theorem~\ref{theorem:SingleComponentCase}, a
description of the transitive sets in $\A_{n}^{k*} \times
\A_{n}^{k*}$ for the diagonal action of $\autgp$.

\section{\bf{Description of Orbit of Pairs for an Ideal}}
\label{sec:OrbitOfPairs} Let $I \in J(\mcl{P})_{\gl}$ be an ideal.
Let the orbit $\grpp^{I^*}$ corresponding to the  ideal $I \in
J(\mcl{P})_{\gl}$ be the box set of the form:
\begin{equation}
\label{eq:Idealboxset} \us{(\partial_{\gl_l} I,\gl_l) \nin
\max(I)}{\prod}\gp^{\partial_{\gl_l} I}\A_{\gl_l}^{\gr_l}\times
\us{\us{(\partial_{\gl_{i_j}} I=s_{i_j},\gl_{i_j}) \in
\max(I)}{j=1}}{\os{t}{\prod}}\gp^{s_{i_j}}\A_{\gl_{i_j}}^{\gr_{i_j}*}
\end{equation}
where $t$ is the cardinality of max elements corresponding to the
ideal $I$ i.e. the cardinality of $\max(I)$. The following are two
basic observations that are needed to state the following theorem.

For simplicity of notation, for a positive integer $k$, define $e^+(k)$ to be the $k$-tuple
$(1,0,....,0)$ and $f^+(k)$ to be the $k$-tuple $(1,1,\ldots,1)$. Then, we observe:
\begin{obs}
\label{obs:CharFormElement}
(Characteristic forms of an element)

(a) By applying a sequence of automorphisms of $\grpp$ we can reduce
any element to a unique characteristic form
\begin{equation*}
(\gp^{r_{1}}e^+(\gr_1),\gp^{r_{2}}e^+(\gr_2),\ldots,\gp^{r_{i}}e^+(\gr_i),\ldots,\gp^{r_{k}}e^+(\gr_k))^{tr}
\end{equation*}
such that $r _{i+1} \leq r _i \leq r _{i+1} + \gl _i - \gl _{i+1} \ \text{ for all } \ i$.

(b) For the abelian group $\grpp$ we can reduce any element to an alternative characteristic form namely
\begin{equation*}
(\gp^{r_{1}}f^+(\gr_1),\gp^{r_{2}}f^+(\gr_2),\gp^{r_{3}}f^+(\gr_3),\ldots,\gp^{r_{k}}f^+(\gr_k))^{tr}
\end{equation*}
such that $r_{i+1} \leq r_i \leq r_{i+1} + \gl_i - \gl_{i+1} \
\text{ for all } \ i$. (See Birkhoff~\cite{GarrettBirkhoff01011935}
and Miller~\cite{1905}). Here $tr$ represents transpose of a vector.
\end{obs}
\begin{obs}
\label{obs:Bijection} Let $G$ be a group. Let $X,Y$ be two $G$-sets.
Let $x_i \in X_i$ and $y_j \in Y_j$ be a collection of
representatives in their transitive subset-partitions $X =
\us{i}{\bigsqcup} X_i$ and $Y = \us{j}{\bigsqcup} Y_j$ respectively.
Let $G_{x_i}$ and $G_{y_j}$ be their stabilizer subgroups in $G$
respectively. Then there are natural bijections among the following.
\begin{itemize}
\item The set of $G$ orbits in $X\times Y$ under the diagonal action of $G$ which contain
$x_i$ in the first coordinate for some element in those orbits with the set of $G_{x_i}$ orbits in $Y$
\item The set of $G$ orbits in $X\times Y$ under the diagonal action of $G$ which contain
$y_j$ in the second coordinate for some element in those orbits with the set of $G_{y_j}$ orbits in $X$
\item The set of double cosets of $G_{x_i}$ and $G_{y_j}$ in $G$ with the
$G$\nobreakdash-orbits of $X_i \times Y_j$.
\item The set of double cosets of $G_{x_i}$ and $G_{y_j}$ in $G$ for all $i,j$
with the $G$\nobreakdash-orbits of $X \times Y$
\end{itemize}
\end{obs}
The theorem below describes transitive sets in similar orbit of pairs.
Recall from observation $24$ that each element of $R_{\gl}$ can be transformed by automorphisms to a characteristic form. For a positive integer $k$,
let us denote by $e^+(k)$,the $k$-tuple $(1,0,....,0)$ and, for $k>1$, by $e^-(k)$, the $(k-1)$-tuple $(1,0,....,0)$. Then, we state our main theorem
in four parts, the first of which is:
\begin{theorem}
\label{theorem:CanonicalSimilarPair} Let $I\in J(\mcl{P})_{\gl}$ be
an ideal with
\begin{equation*}
  \max I = \{(s_{i_j},\gl_{i_j})\mid j = 1 \text{ to } t\}.
\end{equation*}
Denote by $\gl^{'}$ and $\gl^{''}$ be the partitions associated to the
finite modules which arise in the
decomposition(~\ref{eq:decomposition}) of $\grpp$ with respect to
the ideal $I$. Consider any $\autgp^I-$orbit $\mcl{O}_{m',J,K} = \{(l',l'') \in \grppp \oplus
\grpppp = \grpp \mid l' \in m' + \grppp^{J\cup K},l'' \in
\grpppp^{{K}^*}+\grpppp^{J}, I(\bar m')=J \} \subs \grpp^{I^*} \subs
\grpp$. Let $x\in \mcl{O}_{[J\cup
K]_{\gl^{''}}} \subs \grpppp$ be the characteristic element (refer
part $(a)$ of Observation~\ref{obs:CharFormElement}) given by
\begin{equation*}
x = (\gp^{r_{1}}e^{\pm}(\gr_1),\gp^{r_{2}}e^{\pm}(\gr_2),\ldots,\gp^{r_{k}}e^{\pm}(\gr_k))^{tr} \in \mcl{O}_{[J\cup K]_{\gl^{''}}} \subs R_{\gl^{''}}
\end{equation*}
(tr for Transpose) which has the property that $x$ excludes the coordinates that occur
in $\grppp$ and if $\gr_i > 1$ then $r_i =
\partial_{\gl_i}([J\cup K]_{\gl^{''}})$.

Then there exists $y = (y_{i_1},y_{i_2},y_{i_3},\ldots,y_{i_t}) \in
(\A^*)^{tr}$ such that (after permuting coordinates with respect to the decompostion of
$\gl=\gl'\oplus \gl''$)
\begin{equation*}
    ((\gp^{s_{i_1}}y_{i_1},\gp^{s_{i_2}}y_{i_2},\ldots,\gp^{s_{i_t}}y_{i_t}),x) \in O_{m',J,K} \subs \grpp^{I^*}
\end{equation*}
and the transitive set $\mcl{O}$ of $\grpp^{I^*} \times \grpp^{I^*}$ corresponding to
$\mcl{O}_{m',J,K}$ (refer Observation~\ref{obs:Bijection}) is
\begin{equation*}
\mcl{O} = \{(\ul{a},\ul{b})=(ge=ge(I),gf)\in \grpp^{I^*} \times
\grpp^{I^*} \subs \grpp\times\grpp\mid g \in \autgp\}.
\end{equation*}
\end{theorem}
\begin{proof}
Consider the transitive set $\mcl{O}_{m',J,K} \subs \grpp^{I^*}$.
Let $m' = (m_{i_1}',m_{i_2}',\ldots,m_{i_t}') =
(\gp^{n_{i_1}}y_{i_1}',\gp^{n_{i_2}}y_{i_2}',\ldots,\gp^{n_{i_t}}y_{i_t}')$
where $val(m_{i_l}') = n_{i_l}$. We modify $m'$ as follows. First of
all, using Theorems~\ref{theorem:stabilizer}
and~\ref{theorem:sub-orbit}, we note that $m_{i_l}'$ can be changed
to any new element $\ti{m}_{i_l}$ in the coset $m_{i_l}' +
\gp^{\partial_{\gl_{i_l}}([J]_{\gl^{'}}\cup
[K]_{\gl^{'}})}\A_{\gl_{i_l}}$. Here we note that if $n_{i_l} >
s_{i_l}$ then $\gr_{i_l}
> 1$. This is because if $\gr_{i_l}=1$ then since $(s_{i_l},\gl_{i_l}) \in max(I)$ this coordinate should have exact valuation equal to $s_{i_l}$ for all
the elements of the orbit corresponding to $I$. So $\gr_{i_l}>1$.
Now again since $n_{i_l}>s_{i_l}$ and the valuation $s_{i_l}$ must
occur somewhere again because $(s_{i_l},\gl_{i_l}) \in max(I)$. So we should have $s_{i_l} = r_{i_l} =
\partial_{\gl_{i_l}}([J]_{\gl^{''}}\cup [K]_{\gl^{''}})=\partial_{\gl_{i_l}}([J\cup K]_{\gl^{''}})=\partial_{\gl_{i_l}}([J\cup K]_{\gl})=\partial_{\gl_{i_l}}([J\cup K]_{\gl^{'}})=
\partial_{\gl_{i_l}}([J]_{\gl^{'}}\cup [K]_{\gl^{'}})$. So we
modify $m_{i_l}'$ in this coset to get a new element $\ti{m}_{i_l}'$
with valuation $s_{i_l}$. So $\ti{m}' =
(\ti{m}_{i_1}',\ti{m}_{i_2}',\ldots,\ti{m}_{i_t}') =
(\gp^{s_{i_1}}y_{i_1},\gp^{s_{i_2}}y_{i_2},\ldots,\gp^{s_{i_t}}y_{i_t})$
for some unit vector $y =
(y_{i_1},y_{i_2},y_{i_3},\ldots,y_{i_t}) \in (\A^*)^{tr}$. So if we can
take $f = \ti{m}' \oplus x \in \grppp \oplus \grpppp = \grpp$ then
$f \in \mcl{O}_{m',J,K} \subs \grpp^{I^*}$ and $\mcl{O} =
\{(\ul{a},\ul{b})=(ge=ge(I),gf)\in \grpp^{I^*} \times \grpp^{I^*}
\subs \grpp\times\grpp\mid g \in \autgp\}$ is the transitive set
corresponding to the $\autgp^I$-transitive set $\mcl{O}_{m',J,K}$ by
Observation~\ref{obs:Bijection}.
\end{proof}
Using the notation of the previous Theorem~\ref{theorem:CanonicalSimilarPair} let
\begin{equation*}
\mcl{O} = \{(\ul{a},\ul{b})=(ge=ge(I),gf)\in \grpp^{I^*} \times
\grpp^{I^*} \subs \grpp\times\grpp\mid g \in \autgp\}
\end{equation*}
be the transitive set corresponding to the $\autgp^I$-transitive set $\mcl{O}_{m',J,K}$.
Let $S = \{i_1,i_2,\ldots,i_t\}$ where $\max I=\{(s_{i_l},\gr_{i_l})\in J(\mcl{P})_{\gl}\mid l=1,2,\ldots,t\}$.
Let
\begin{equation*}
\begin{aligned}
m_{i_l} =& min\bigg(
\us{j<i_l,j\in S}{min}(s_j+val(y_j-y_{i_l})),\\
& \us{j>i_l,j\in S}{min}(\gl_{i_l}-\gl_j+s_j+val(y_j-y_{i_l})),\\
& \us{j<i_l,j \in S,\gr_j>1}{min}(r_j), \us{j<i_l,j \nin S}{min}(r_j),\\
& \us{j>i_l,j \in S,\gr_j>1}{min}(\gl_{i_l}-\gl_j+r_j), \us{j>i_l,j \nin S}{min}(\gl_{i_l}-\gl_j+r_j)\bigg)
\end{aligned}
\end{equation*}
Again for simplicity of notation for an integer $k>0, x,y \in \AAA {\gm}$
define $t(x,y;k)$ to be the k-tuple $(x,y,0,...,0)$.
\begin{theorem}
\label{theorem:GeneralOrbitPairMax}
With these notations we have
\begin{itemize}
\item $a_{i_l},b_{i_l} \in (\gp^{s_{i_l}}\A_{\gl_{i_l}}^{\gr_{i_l}*})$.
\item $m_{i_l} > s_{i_l}$ and exactly one of the following holds.
\begin{enumerate}[label=\Alph*$\bm{.}$]
\item   $(b_{i_l}-y_{i_l}a_{i_l}) \in \gp^{m_{i_l}}\A_{\gl_{i_l}}^{\gr_{i_l}}$ if
$m_{i_l} \leq r_{i_l}, \gr_{i_l} > 1$ or if $\gr_{i_l} = 1$.
\item   $(b_{i_l}-y_{i_l}a_{i_l}) \in \gp^{r_{i_l}}\A_{\gl_{i_l}}^{\gr_{i_l}}$ if
$m_{i_l} > r_{i_l}, \gr_{i_l} > 1$ and $\gp^{-r_{i_l}}(b_{i_l}-y_{i_l}a_{i_l}) (mod\ \ \gp)$
is linearly independent with $\gp^{-s_{i_l}}a_{i_l} (mod\ \ \gp)$ in $\mbb{F}_q^{\gr_{i_l}}$.
\end{enumerate}
\end{itemize}
\end{theorem}
\begin{proof}
Here we describe the max-components of the elements in the orbit of pairs containing $(e,f)$. Let $g \in \autgp$ be an element described as $g_{mat}$ below.
\begin{equation}
\label{eq:aut}
g_{mat}=
\left(
\begin{array}{cccc}
 A_{11} & A_{12} \gp^{\lambda _1-\lambda _2} & \cdots & A_{1k} \gp^{\lambda _1-\lambda _k} \\
 A_{21} & A_{22} & \cdots & A_{2k} \gp^{\lambda _2-\lambda _k} \\
 \vdots & \vdots & \ddots & \vdots \\
 A_{k1} & A_{k2} & \cdots & A_{kk}
\end{array}
\right)
\end{equation}
where each $A_{ij}$ is a $\gr_i \times \gr_j$ matrix of elements from $\A$ and $det(A_{ii})$ is a unit in $\A$.

Applying $g_{mat}$ we get the following equations for
$(\ul{a},\ul{b})=g_{mat}(e=e(I),f)$ we have for $1\leq l \leq t$
\begin{equation}
\label{eq:A}
\begin{aligned}
a_{i_l} &= \us{j<{i_l},j\in S}{\sum}A_{i_lj}\gp^{s_j}e^+(\gr_j) + A_{i_li_l}\gp^{s_{i_l}}e^+(\gr_{i_l})
+ \us{j>{i_l},j\in S}{\sum}A_{i_lj}\gp^{\gl_{i_l}-\gl_j+s_j}\gp^{s_j}e^+(\gr_j)
\end{aligned}
\end{equation}
\begin{equation}
\label{eq:B}
\begin{aligned}
&b_{i_l} = \us{j<{i_l},j\in S}{\sum}A_{i_lj}\gp^{s_j}t(y_j,\gp^{r_j-s_j};\gr_j)
        + A_{i_li_l}\gp^{s_{i_l}}t(y_{i_l},\gp^{r_{i_l}-s_{i_l}};\gr_{i_l}) +\\
        & \us{j>{i_l},j\in S}{\sum}A_{i_lj}\gp^{\gl_{i_l}-\gl_j+s_j}t(y_j,\gp^{r_j-s_j};\gr_j)
        + \us{j<{i_l},j\nin S}{\sum}A_{i_lj}\gp^{r_j}e^+(\gr_j)+\us{j>{i_l},j\nin S}{\sum}A_{i_lj}\gp^{\gl_{i_l}-\gl_j+r_j}e^+(\gr_j)
\end{aligned}
\end{equation}
We conclude that because $f \in \grpp^{I^*}$ and from the structure of the
box set of $\grpp^{I^*}$ in equation~(\ref{eq:Idealboxset}) we get
that for any $i \in S = \{i_1,i_2,\ldots,i_t\}$
\begin{equation}
\label{eq:InequalityChain}
\begin{aligned}
&s_{i_1} > s_{i_2} > \ldots > s_{i_t}\\
&\gl_{i_1} - s_{i_1} > \gl_{i_2} - s_{i_2} > \ldots > \gl_{i_t} - s_{i_t}\\
&r_j \geq s_j > s_i \text{ for all }j < i, j \nin S \text{ and for all }j < i, j \in S \text{ if }\gr_{j} > 1\\
&\gl_i-\gl_j+r_j \geq \gl_i-\gl_j+s_j > s_i \text{ for all } j > i, j \nin S \text{ and for all } j > i, j \in S \text{ if } \gr_{j} > 1
\end{aligned}
\end{equation}
Hence from equations~(\ref{eq:A}) and (\ref{eq:B}) we have that
$a_{i_l}, b_{i_l} \in \gp^{s_{i_l}}\A_{\gl_{i_l}}^{{\gr_i}^*}$
(which automatically holds because $\ul{a},\ul{b} \in \grpp^{I^*}$).
In addition we also have
\begin{equation}
\label{eq:C}
\begin{aligned}
b_{i_l}-y_{i_l}a_{i_l} &= \us{j<{i_l},j\in S}{\sum}A_{i_lj}\gp^{s_j}t(y_j-y_{i_l},\gp^{r_j-s_j};\gr_j)
+ A_{i_li_l}\gp^{s_{i_l}}t(0,\gp^{r_{i_l}-s_{i_l}};\gr_{i_l})\\
&+\us{j>{i_l},j\in S}{\sum}A_{i_lj}\gp^{\gl_{i_l}-\gl_j+s_j}t(y_j-y_{i_l},\gp^{r_j-s_j};\gr_j)\\
&+\us{j<{i_l},j\nin S}{\sum}A_{i_lj}\gp^{r_j}e^+(\gr_j) + \us{j>{i_l},j\nin S}{\sum}A_{i_lj}\gp^{\gl_{i_l}-\gl_j+r_j}e^+(\gr_j)
\end{aligned}
\end{equation}
If we define $m_{i_l}$ as in Theorem~\ref{theorem:GeneralOrbitPairMax} we conclude that from inequalities(~\ref{eq:InequalityChain}), $m_{i_l} > s_{i_l}$ and also exactly one of the following holds.
\begin{itemize}
\item   $b_{i_l}-y_{i_l}a_{i_l} \in \gp^{m_{i_l}}\A_{\gl_{i_l}}^{\gr_{i_l}}$ if $m_{i_l} \leq r_{i_l}, \gr_{i_l} > 1$ or if $\gr_{i_l} = 1$.
\item   $b_{i_l}-y_{i_l}a_{i_l} \in \gp^{r_{i_l}}\A_{\gl_{i_l}}^{\gr_{i_l}}$ if $m_{i_l} > r_{i_l}, \gr_{i_l} > 1$ and $\gp^{-r_{i_l}}(b_{i_l}-y_{i_l}a_{i_l}) (mod\ \ \gp)$ is linearly independent with $\gp^{-s_{i_l}}a_{i_l} (mod\ \ \gp)$ in $\mbb{F}_q^{\gr_{i_l}}$.
\end{itemize}
\end{proof}

Using the notation of the Theorem~\ref{theorem:CanonicalSimilarPair} let
\begin{equation*}
\mcl{O} = \{(\ul{a},\ul{b})=(ge=ge(I),gf)\in \grpp^{I^*} \times
\grpp^{I^*} \subs \grpp\times\grpp\mid g \in \autgp\}
\end{equation*}
be the transitive set corresponding to the $\autgp^I$-transitive set $\mcl{O}_{m',J,K}$.
Let $S = \{i_1,i_2,\ldots,i_t\}$ where $\max I=\{(s_{i_l},\gr_{i_l})\in J(\mcl{P})_{\gl}\mid l=1,2,\ldots,t\}$.
If $i \nin S$, there exists an $i_l \in S,\partial_{\gl_i}I = s_{i_l}$ or
$\partial_{\gl_i}I = \gl_i-\gl_{i_l}+s_{i_l}$ such that $a_i,b_i \in
\gp^{\partial_{\gl_i}I}\A_{\gl_i}^{\gr_i}$.
Let
\begin{equation*}
\begin{aligned}
        m_i &= min\bigg(\us{j<i,j\in S}{min}(s_j+val(y_j-y_{i_l})),\\
                & \us{j>i,j\in S}{min}(\gl_i-\gl_j+s_j+val(y_j-y_{i_l})),\\
                & \us{j<i,j \in S,\gr_j>1}{min}(r_j), \us{j<i,j \nin S}{min}(r_j),\\
                & \us{j>i,j \in S,\gr_j>1}{min}(\gl_i-\gl_j+r_j), \us{j>i, j \nin S}{min}(\gl_i-\gl_j+r_j)\bigg)
\end{aligned}
\end{equation*}
\begin{theorem}
\label{theorem:GeneralOrbitPairNonMax}
With these notations we have $m_i \geq \partial_{\gl_i}I$ and exactly one of the
following holds.
\begin{enumerate}[label=(\roman*)$\bm{.}$]
\item $b_i-y_{i_l}a_i \in \gp^{m_i}\A_{\gl_i}^{\gr_i}$ if $r_i \geq m_i > \partial_{\gl_i}I$
\item $b_i-y_{i_l}a_i \in (\gp^{\partial_{\gl_i}([J]_{\gl^{''}}\cup [K]_{\gl^{''}})}
      \A_{\gl_i}^{\gr_i*})$
with \\ $m_i > r_i=\partial_{\gl_i}([J]_{\gl^{''}}\cup [K]_{\gl^{''}}) \geq \partial_{\gl_i}I$
\item $(a_i,b_i) \in \gp^{\partial_{\gl_i}I}\A_{\gl_i}^{\gr_i} \oplus
            \gp^{\partial_{\gl_i}I}\A_{\gl_i}^{\gr_i}$ can be any element.
\end{enumerate}
\end{theorem}
\begin{proof}
Here we describe the nonmax-components of the elements in the orbit of pairs containing $(e,f)$. Let $g \in \autgp$ be an element described as $g_{mat}$ below.
\begin{equation}
\label{eq:aut1}
g_{mat}=
\left(
\begin{array}{cccc}
 A_{11} & A_{12} \gp^{\lambda _1-\lambda _2} & \cdots & A_{1k} \gp^{\lambda _1-\lambda _k} \\
 A_{21} & A_{22} & \cdots & A_{2k} \gp^{\lambda _2-\lambda _k} \\
 \vdots & \vdots & \ddots & \vdots \\
 A_{k1} & A_{k2} & \cdots & A_{kk}
\end{array}
\right)
\end{equation}
where each $A_{ij}$ is a $\gr_i \times \gr_j$ matrix of elements from $\A$ and $det(A_{ii})$ is a unit in $\A$.

Applying $g_{mat}$ we get the following equations for
$(\ul{a},\ul{b})=g_{mat}(e=e(I),f)$:

we have for $i \nin S = \{i_1 < i_2 <\ldots <i_t\} \subs \{1,2,3,\ldots ,k\}$
\begin{equation}
\label{eq:D}
a_{i} = \us{j<{i},j\in S}{\sum}A_{ij}\gp^{s_j}e^+(\gr_j) + \us{j>{i},j\in S}{\sum}A_{ij}\gp^{\gl_{i}-\gl_j+s_j}e^+(\gr_j)
\end{equation}
\begin{equation}
\label{eq:E}
\begin{aligned}
b_{i} &= \us{j<{i},j\in S}{\sum}A_{ij}\gp^{s_j}t(y_j,\gp^{r_j-s_j};\gr_j) +  \us{j>{i},j\in S}{\sum}A_{ij}\gp^{\gl_{i}-\gl_j+s_j}t(y_j,\gp^{r_j-s_j};\gr_j)\\
&+ \us{j<{i},j\nin S}{\sum}A_{ij}\gp^{r_j}e^+(\gr_j) + A_{ii}\gp^{r_i}e^+(\gr_i) + \us{j>{i},j\nin S}{\sum}A_{ij}\gp^{\gl_{i}-\gl_j+r_j}e^+(\gr_j)
\end{aligned}
\end{equation}
Suppose $i_l < i < i_{l+1}$ if such an $i_l$ and $i_{l+1}$ exist (otherwise either $i < i_{l+1} = i_{1}$ and $i_l$ does not exist
or $i_{l} = i_t < i$ and $i_{l+1}$ does not exist). Then we have the following cases.
\begin{enumerate}[label=\arabic*$\bm{.}$]
\item $s_{i_l} < \gl_i-\gl_{i_{l+1}}+s_{i_{l+1}},\partial_{\gl_i}I = s_{i_l}$
\item $s_{i_l} > \gl_i-\gl_{i_{l+1}}+s_{i_{l+1}},\partial_{\gl_i}I = \gl_i-\gl_{i_{l+1}}+s_{i_{l+1}}$
\item   if $i < i_1 = i_{l+1}, \partial_{\gl_i}I = \gl_i-\gl_{i_{l+1}}+s_{i_{l+1}}$
\item if $i > i_t = i_l, \partial_{\gl_i}I = s_{i_l}$
\item $s_{i_l} = \gl_i-\gl_{i_{l+1}}+s_{i_{l+1}}=\partial_{\gl_i}I$
\end{enumerate}
First we note that $a_i,b_i \in \gp^{\partial_{\gl_i}I}\A_{\gl_i}^{\gr_i}$. Let $m_i$ be as defined in the Theorem~\ref{theorem:GeneralOrbitPairNonMax}.
Now we analyze the valuations of the summands in the Equations~\ref{eq:D},\ref{eq:E} as given in the following cases to arrive at conclusions.

\fo{15}{15}{\bf{Cases $1,2,3,4:$}} $\partial_{\gl_i}I = s_{i_l}$ or $\partial_{\gl_i}I = \gl_i-\gl_{i_{l+1}}+s_{i_{l+1}}$\\
Let $l_0$ be any element having the following property.
\begin{itemize}
\item $l_0 \leq i$ and $l_0 \nin S$ such that $r_{l_0} = \partial_{\gl_i}I$
\item $l_0 < i$ and $l_0 \in S$ with $\gr_{l_0} > 1$ such that $r_{l_0} = \partial_{\gl_i}I$
\item $l_0 > i$ and $l_0 \nin S$ such that $\gl_i-\gl_{l_0}+r_{l_0} = \partial_{\gl_i}I$
\item $l_0 > i$ and $l_0 \in S$ with $\gr_{l_0} > 1$ such that $\gl_i-\gl_{l_0}+r_{l_0} = \partial_{\gl_i}I$
\end{itemize}
\begin{enumerate}
\item If there does not exist any such $l_0$ then
    \begin{itemize}
    \item we conclude that $m_i > \partial_{\gl_i}I$
    \item $b_i-y_{i_l}a_i \in \gp^{m_i}\A_{\gl_i}^{\gr_i}$ if $r_i \geq m_i$
    \item $b_i-y_{i_l}a_i \in \gp^{r_i}\A_{\gl_i}^{\gr_i*}$ if $r_i < m_i$
    \end{itemize}
\item If there exist unique such $l_0$ then
    \begin{itemize}
    \item If $l_0 = i$ and $r_{l_0}=r_i$ then
        \begin{enumerate}[label=$\bullet$]
            \item $ m_i > \partial_{\gl_i}I = r_{l_0} = r_i$.
            \item $b_i-y_{i_l}a_i \in \gp^{r_i}\A_{\gl_i}^{\gr_i*}$
        \end{enumerate}
    \item If $l_0 \neq i$ then $(a_i,b_i)$ can be any element of $\gp^{\partial_{\gl_i}I}\A_{\gl_i}^{\gr_i} \oplus \gp^{\partial_{\gl_i}I}\A_{\gl_i}^{\gr_i}$.
    \end{itemize}
\item If there exist more than one $l_0$ then $(a_i,b_i)$ can be any element of $\gp^{\partial_{\gl_i}I}\A_{\gl_i}^{\gr_i} \oplus
\gp^{\partial_{\gl_i}I}\A_{\gl_i}^{\gr_i}$.
\end{enumerate}
\fo{15}{15}{\bf{Case $5:$}}$s_{i_l} = \gl_i-\gl_{i_{l+1}}+s_{i_{l+1}}=\partial_{\gl_i}I$\\
In this case $(a_i,b_i)$ can be any element of
$\gp^{\partial_{\gl_i}I}\A_{\gl_i}^{\gr_i} \oplus
\gp^{\partial_{\gl_i}I}\A_{\gl_i}^{\gr_i}$.
\end{proof}
This completes half the proof which involves the description of the
transitive subset containing $(e=e(I),f)$ in $\grpp^{I^*} \times
\grpp^{I^*}$.

Using the notation of Theorem~\ref{theorem:CanonicalSimilarPair}, let
\begin{equation*}
\mcl{O} = \{(\ul{a},\ul{b})=(ge=ge(I),gf)\in \grpp^{I^*} \times
\grpp^{I^*} \subs \grpp\times\grpp\mid g \in \autgp\}
\end{equation*}
be the transitive set corresponding to the $\autgp^I$-transitive set $\mcl{O}_{m',J,K}$.
\begin{theorem}
\label{theorem:GeneralOrbitPairConverse}
Suppose the element $(\ul{a},\ul{b})\in \grpp^{I^*} \times
\grpp^{I^*} $ has the same case description as that of $(e(I),f)$ then there exists an element
$g \in \autgp$ such that $(\ul{a},\ul{b})=(ge(I),gf)$ i.e. $(\ul{a},\ul{b}) \in \mcl{O}$.
\end{theorem}
\begin{proof}
Now we look at the converse. Let $I$ be the ideal with its
corresponding orbit $\grpp^{I^*}$, and associated partitions
$\gl^{'},\gl^{''}$. Let $\mcl{O}_{m',J,K}$ be the
$\autgp^I$-transitive set and $y \in (\A^*)^k$ be the unit
vector. These are all given as defined in
Theorem~\ref{theorem:CanonicalSimilarPair}. Let $(e,f), m_i, i\in
\{1,2,...,k\}$ be also as defined in the
Theorems~\ref{theorem:GeneralOrbitPairMax} and \ref{theorem:GeneralOrbitPairNonMax}. Let $(\ul{a},\ul{b}) \in
\grpp^{I^*} \times \grpp^{I^*}$ such that for each $1 \leq i \leq
k$, similar to $(e_i,f_i), (a_i,b_i)$ also satisfies the same one of
the cases with the conditions given in the hypothesis of these cases
then we observe that there exists a $g \in \autgp$ such that
$(ge(I),gf) = (\ul{a},\ul{b})$. The construction of an invertible
matrix $g \in \autgp$ is done in each block row. The conditions are
such that we can perform this construction independently in each
block row using appropriate valuations and linearly independent
conditions.

Let $g \in \autgp$ be an element as described in equation~(\ref{eq:aut}).

Suppose there exists $i_l \in S$ such that $s_{i_l} < m_{i_l} \leq
r_{i_l}$ and the minimum is attained at $s_j + val(y_j-y_{i_l})$ for
some $j < i_l, j \in S$. Also suppose $(\ul{a},\ul{b})$ satisfies
the hypothesis of this condition i.e. $b_{i_l}-a_{i_l}y_{i_l} \in
\gp^{m_{i_l}}\A_{\gl_i}^{\gr_{i_l}}$. This occurs in Case $A$ of
Theorem~\ref{theorem:GeneralOrbitPairMax}. We determine the $i_l^{th}$
block row of $g \in \autgp$ as follows. Set $A_{i_lp} = 0$ for $p
\neq i_l$ and $p \neq j$. To determine $A_{i_lj}$ and $A_{i_li_l}$
we solve the equations.
\begin{equation*}
\begin{aligned}
& A_{i_lj}\gp^{s_j}e^+(\gr_j) + A_{i_li_l}\gp^{s_{i_l}}e^+(\gr_{i_l}) = a_{i_l}.\\
& A_{i_lj}\gp^{s_j}t(y_j,\gp^{r_j-s_j};\gr_j)+A_{i_li_l}\gp^{s_{i_l}}t(y_{i_l},\gp^{r_{i_l}-s_{i_l}};\gr_{i_l})
= b_{i_l}.\\
& A_{i_lj}\gp^{s_j}t(y_j-y_{i_l},\gp^{r_j-s_j};\gr_j) + A_{i_li_l}t(0,\gp^{r_{i_l}};\gr_{i_l})
= b_{i_l} - a_{i_l}y_{i_l} = \gp^{m_{i_l}}C\\
& \text{ for some column vector } C \in \A_{\gl_{i_l}}^{\gr_{i_l}}.\\
\end{aligned}
\end{equation*}
Let $(y_j-y_{i_l})\gp^{s_j} = \gp^{m_{i_l}}y'$ for some unit $y' \in \A$. Let $C_{i_lj}^{1},C_{i_lj}^{2},C_{i_li_l}^{1},C_{i_li_l}^{2}$
denote the first and second columns of $A_{i_lj},A_{i_li_l}$ respectively. Choose $C_{i_li_l}^{2}$ to be any vector in $(\AAA {\gl_{i_l}})^{\gr_{i_l}}$
such that $C_{i_li_l}^{2} (mod\ \ \gp)$ is linearly independent from $\gp^{-s_{i_l}}a_{i_l} (mod\ \ \gp)$. Choose $C_{i_lj}^{p} = 0$ for $p>2$.
We get the following equations for the columns.
\begin{equation*}
\begin{aligned}
& \gp^{s_j}C_{i_lj}^{1} + \gp^{s_{i_l}}C_{i_li_l}^{1} = a_{i_l}.\\
& y'\gp^{m_{i_l}}C_{i_lj}^{1} + \gp^{r_j}C_{i_lj}^{2} + \gp^{r_{i_l}}C_{i_li_l}^{2} = \gp^{m_{i_l}}C.
\end{aligned}
\end{equation*}
Choose $C_{i_lj}^{2} = 0$ and solving for $C_{i_lj}^{1}$ and then for $C_{i_li_l}^{1}$ we get
$C_{i_lj}^{1} = (y')^{-1}(C - \gp^{r_{i_l}-m_{i_l}}C_{i_li_l}^{2})$ and $C_{i_li_l}^{1} = \gp^{-s_{i_l}}a_{i_l} - \gp^{s_j-s_{i_l}}C_{i_lj}^{1}$.
Since $s_j > s_{i_l},C_{i_li_l}^{1} \equiv \gp^{-s_{i_l}}a_{i_l} (mod\ \ \gp)$ and is linearly independent from
$C_{i_li_l}^{2} (mod\ \ \gp)$. Now extend the columns of $A_{i_li_l}$ to a matrix such that $A_{i_li_l} (mod\ \ \gp)$ is invertible.

Suppose there exists $i_l \in S$ such that $m_{i_l} > r_{i_l} \geq
s_{i_l}$. Also suppose $(\ul{a},\ul{b})$ satisfies the hypothesis of
this condition i.e. $b_{i_l}-a_{i_l}y_{i_l} \in
\gp^{r_{i_l}}\A_{\gl_{i_l}}^{\gr_{i_l}}$ and
$\gp^{-r_{i_l}}(b_{i_l}-y_{i_l}a_{i_l}) (mod\ \ \gp)$ is linearly
independent with $\gp^{-s_{i_l}}a_{i_l} (mod\ \ \gp)$. This occurs
in Case $B$ of Theorem~\ref{theorem:GeneralOrbitPairMax}. We determine
the $i_l^{th}$ block row of $g \in \autgp$ as follows. Set $A_{i_lp}
= 0$ for $p \neq i_l$. To determine $A_{i_li_l}$ we solve the
equations.
\begin{equation*}
\begin{aligned}
& A_{i_li_l}\gp^{s_{i_l}}e^+(\gr_{i_l}) = a_{i_l}.\\
& A_{i_li_l}\gp^{s_{i_l}}t(y_{i_l},\gp^{r_{i_l}-s_{i_l}};\gr_{i_l}) = b_{i_l}.\\
& A_{i_li_l}t(0,\gp^{r_{i_l}};\gr_{i_l}) = b_{i_l} - a_{i_l}y_{i_l} = \gp^{r_{i_l}}C\\
& \text{ for some column vector } C \in \A_{\gl_{i_l}}^{\gr_{i_l}*}.\\
\end{aligned}
\end{equation*}
Let $C_{i_li_l}^{1},C_{i_li_l}^{2}$ denote the first and second
columns of $A_{i_li_l}$ respectively. Choose $C_{i_li_l}^{1}$ to be
$\gp^{-s_{i_l}}a_{i_l}$ and $C_{i_li_l}^{2}$ to be the vector
$\gp^{-r_{i_l}}(b_{i_l}-y_{i_l}a_{i_l}) \in \A_
{\gl_{i_l}}^{\gr_{i_l}}$. Then we have the linearly independent
condition satisfied for $A_{i_li_l}$ and extend the columns of
$A_{i_li_l}$ such that the matrix $A_{i_li_l} (mod\ \ \gp)$ is
invertible.

Now suppose there exists an $i \nin S$ and $i_l < i < i_{l+1}$ such
that $r_i \geq m_i > \partial_{\gl_i}I$ and the minimum is attained
at $s_j + val(y_j-y_{i_l})$ for some $j < i_l, j \in S$. Also
suppose $(\ul{a},\ul{b})$ satisfies the hypothesis of this condition
i.e. $b_i-a_iy_i \in \gp^{m_i}\A_{\gl_i}^{\gr_i}$. This occurs in
Case $(i)$ of Theorem~\ref{theorem:GeneralOrbitPairNonMax}. We determine the
$i^{th}$ block row of $g \in \autgp$ as follows. Set $A_{ip} = 0$
for $p \neq i_l,j,i$. To determine $A_{ij},A_{ii_l},A_{ii}$ we solve
the equations.
\begin{equation*}
\begin{aligned}
& A_{ij}\gp^{s_j}e^+(\gr_j) + A_{ii_l}\gp^{s_{i_l}}e^+(\gr_{i_l}) = a_{i}.\\
& A_{ij}\gp^{s_j}t(y_j,\gp^{r_j-s_j};\gr_j) + A_{ii_l}\gp^{s_{i_l}}t(y_{i_l},\gp^{r_{i_l}-s_{i_l}};\gr_{i_l}) + A_{ii}\gp^{r_{i}}e^+(\gr_i) = b_{i}.\\
& A_{ij}\gp^{s_j}t(y_j-y_{i_l},\gp^{r_j-s_j};\gr_j) + A_{ii_l}t(0,\gp^{r_{i_l}};\gr_{i_l}) + A_{ii}\gp^{r_{i}}e^+(\gr_i)
= b_{i} - a_{i}y_{i_l} = \gp^{m_i}C\\
&\text{ for some column vector } C \in \A_{\gl_{i}}^{\gr_{i}}.\\
\end{aligned}
\end{equation*}
Let $(y_j-y_{i_l})\gp^{s_j} = \gp^{m_{i}}y'$ for some unit $y' \in
\A$. Let
$C_{ii_l}^{1},C_{ii_l}^{2},C_{i_lj}^{1},C_{i_lj}^{2},C_{ii}^{1}$
denote the first and second columns of $A_{ii_l},A_{ij}$ and first
column of $A_{ii}$ respectively. Set the columns $C_{ij}^{p} =
C_{ii_l}^{p} = 0$ for $p > 2$. Choose $C_{ii}^{1}$ to be any vector
in $\A_{\gl_i}^{\gr_i*}$. Now extend the columns of $A_{ii}$ to a
matrix such that $A_{ii} (mod\ \ \gp)$ is invertible. We get the
following equations for the columns.
\begin{equation*}
\begin{aligned}
& \gp^{s_j}C_{ij}^{1} + \gp^{s_{i_l}}C_{ii_l}^{1}  = a_{i}.\\
& y'\gp^{m_{i}}C_{ij}^{1} + \gp^{r_j}C_{ij}^{2} + \gp^{r_{i_l}}C_{ii_l}^{2} + \gp^{r_i}C_{ii}^1 = \gp^{m_{i}}C.
\end{aligned}
\end{equation*}
Choose $C_{ij}^{2} = C_{ii_l}^{2} = 0$ and solving for $C_{ij}^{1}$ and then for $C_{ii_l}^{1}$ we get
$C_{ij}^{1} = (y')^{-1}(C - \gp^{r_{i}-m_{i}}C_{ii}^{1})$ and $C_{ii_l}^{1} = \gp^{-s_{i_l}}a_{i} -
\gp^{s_j-s_{i_l}}C_{ij}^{1}$.

The rest of the cases are similar. This completes the proof of
Theorem~\ref{theorem:GeneralOrbitPairConverse}.
\end{proof}
\begin{example}
Here is an example that describes the orbit of pairs in the two
component case. Consider the finite module $\A_{l} \oplus \A_{k}$
with $k<l$ and without multiplicity corresponding to the partition
$\gl=(l^1,k^1) \in \Gl$. Consider the orbit $\grpp^{I^*} \subs
\A_{l} \oplus \A_{k}$ of the non-principal ideal $I$ with $\max(I) =
\{(s+r,l),(s,k)\} \in \mcl{P}$. Hence $\grpp^{I^{*}} =
\gp^{s+r}\A_l^* \times \gp^s\A_k^*$ with $0<r<l-k$. The transitive
subsets of $\grpp^{I^*} \times \grpp^{I^*}$ under the diagonal
action $\autgp$ is given as follows. Given two units $y$ and $x$ in
$\A^{*}$, the transitive subset
\begin{equation*}
\begin{aligned}
&I_{y,x} =
\{((\gp^{r+s}u,\gp^{s}w),(\gp^{r+s}u^{'},\gp^{s}w^{'}))\mid
u,u^{'},w,w^{'} \text{ are units and }\\
&(\gp^{r+s}u^{'}-\gp^{r+s}uy,\gp^{s}w^{'}-\gp^{s}wx) \in
\gp^{l-k+s+t}\A_{l} \oplus \gp^{r+s+t}\A_{k}\}\text{ where }t \|
(x-y).
\end{aligned}
\end{equation*}
Similarly the ordered pairs $\{(y,x)\in \A_{l-k-r+t}^{*} \oplus
\A_{r+t}^{*}\}$ is the parameter group which is independent of the
shift parameter $s$.
\end{example}

\section{\bf{Commutativity}}
Let $\grpp^{I^*} \subs \grpp$ be an orbit under the action of
$\autgp$. The multiplication in the endomorphism algebra
$\Endo_{\autgp}(\mbb{C}[\grpp^{I^*}])$ is given by convolution. We
prove commutativity of this convolution here.

For $j=1,2$ let $\mcl{O}_j \subs \grpp^{I^*} \times \grpp^{I^*}$
denote two transitive sets. Let $\mcl{I}_{\mcl{O}_1}$ and
$\mcl{I}_{\mcl{O}_2}$ denote the indicator functions of these two
orbits. Suppose $(\ga,\gb) \in \grpp^2$ be an element. Then
$\mcl{I}_{\mcl{O}_1}*\mcl{I}_{\mcl{O}_2}(\ga,\gb) = \us{\gga \in
\grpp}{\sum}\mcl{I}_{\mcl{O}_1}(\ga,\gga)\mcl{I}_{\mcl{O}_2}(\gga,\gb)$
and $\mcl{I}_{\mcl{O}_2}*\mcl{I}_{\mcl{O}_1}(\ga,\gb) = \us{\gd \in
\grpp}{\sum}\mcl{I}_{\mcl{O}_2}(\ga,\gd)\mcl{I}_{\mcl{O}_1}(\gd,\gb)$.
To prove commutativity we need to prove that the existence of an
element $\gga \in \grpp$ such that $\mcl{I}_{\mcl{O}_1}(\ga,\gga) =
1 = \mcl{I}_{\mcl{O}_2}(\gga,\gb)$ is equivalent to the existence of
an element $\gd \in \grpp$ such that $\mcl{I}_{\mcl{O}_2}(\ga,\gd) =
1 = \mcl{I}_{\mcl{O}_1}(\gd,\gb)$ and that the number of solutions
for $\gga$ to the equations
\begin{equation}
\begin{aligned}
\mcl{I}_{\mcl{O}_1}(\ga,\gga) &= 1\\
\mcl{I}_{\mcl{O}_2}(\gga,\gb) &= 1
\end{aligned}
\end{equation}
is equal to the number of solutions for $\gd$ to the equations
\begin{equation}
\begin{aligned}
\mcl{I}_{\mcl{O}_2}(\ga,\gd) &= 1\\
\mcl{I}_{\mcl{O}_1}(\gd,\gb) &= 1
\end{aligned}
\end{equation}
We prove this componentwise for $\mcl{I}_{\mcl{O}_1}$ and
$\mcl{I}_{\mcl{O}_2}$ corresponding to each isotypic component
$\A_{\gl_i}^{\gr_i}$ of $\grpp$.

First we prove a simple lemma on counting number of solutions to certain congruences with certain conditions.
\begin{lemma}
\label{lemma:LemmaCommutativity} Let $y,y^{'} \in \mcl{A}_{n,k} =
\A_{n}^k$ and $r,r^{'} \in \{0,1,2,\ldots ,n\}$. Let
$\mcl{A}_{n,k}^{*}$ denote the set $\mcl{A}_{n,k}-\gp\mcl{A}_{n,k}$.
Then
\begin{enumerate}
\item $|(y+\gp^{r}\mcl{A}_{n,k}^{*})\cap(y^{'}+\gp^{r^{'}}\mcl{A}_{n,k}^{*})| = |(y^{'}+\gp^{r}\mcl{A}_{n,k}^{*})\cap(y+\gp^{r^{'}}\mcl{A}_{n,k}^{*})|$\\
If the intersection of the left hand side is non-empty then

$|((y+\gp^{r}\mcl{A}_{n,k}^{*})\cap(y^{'}+\gp^{r^{'}}\mcl{A}_{n,k}^{*}))| =$\\
\begin{center}
$\begin{cases}
    q^{(n-r-1)k}(q^k-1) & \text{ for }r>r^{'}\\
    q^{(n-r-1)k}(q^k-1) & \text{ if } y-y' \in \gp^{(r+1)}\mcl{A}_{n,k},r=r^{'}\\
    q^{(n-r-1)k}(q^k-2) & \text{ if } y-y' \in \gp^{r}\mcl{A}_{n,k}^{*},r=r^{'}
\end{cases}$
\end{center}
\item $|(y+\gp^{r}\mcl{A}_{n,k})\cap(y^{'}+\gp^{r^{'}}\mcl{A}_{n,k})| = |(y^{'}+\gp^{r}\mcl{A}_{n,k})\cap(y+\gp^{r^{'}}\mcl{A}_{n,k})|$\\
If the intersection of the left hand side is non-empty then

$|((y+\gp^{r}\mcl{A}_{n,k})\cap(y^{'}+\gp^{r^{'}}\mcl{A}_{n,k}))| = q^{(n-r)k}$.

\item $|(y+\gp^{r}\mcl{A}_{n,k})\cap(y^{'}+\gp^{r^{'}}\mcl{A}_{n,k}^{*})| = |(y^{'}+\gp^{r}\mcl{A}_{n,k})\cap(y+\gp^{r^{'}}\mcl{A}_{n,k}^{*})|$

If the intersection of the left hand side is non-empty then

$|(y+\gp^{r}\mcl{A}_{n,k})\cap(y^{'}+\gp^{r^{'}}\mcl{A}_{n,k}^{*})| =$\\
\begin{center}
$\begin{cases}
    q^{(n-r)k} & \text{ for }r>r^{'}\\
    q^{(n-r'-1)k}(q^k-1) & \text{ if } r\leq r^{'}\\
\end{cases}$
\end{center}
\end{enumerate}
\end{lemma}
\begin{proof}
Suppose $r \geq r^{'}$ we produce a bijection between the sets $(y+\gp^{r}\mcl{A}_{n,k}^{*})\cap(y^{'}+\gp^{r^{'}}\mcl{A}_{n,k}^{*})$ and $(y^{'}+\gp^{r}\mcl{A}_{n,k}^{*})\cap(y+\gp^{r^{'}}\mcl{A}_{n,k}^{*})$ as follows. Let
\begin{center}
$\underset{x=y+\gp^{r}a_{1} = y^{'}+\gp^{r^{'}}a_{2}}{\ul{(y+\gp^{r}\mcl{A}_{n,k}^{*})\cap(y^{'}+\gp^{r^{'}}\mcl{A}_{n,k}^{*})}} \llra \underset{(a_{1},a_{2})}{\ul{B \subs \mcl{A}_{n-r,k}^{*}\times \mcl{A}_{n-r^{'},k}^{*}}} \lla$\\
$\lra \underset{(b_{1},b_{2}) = (a_{1}, -a_{2}+2 a_{1} \gp^{r-r^{'}})}{\ul{B^{'} \subs \mcl{A}_{n-r,k}^{*} \times \mcl{A}_{n-r^{'},k}^{*}}} \llra \underset{z = y^{'}+\gp^{r}b_{1} = y+\gp^{r^{'}}b_{2}}{\ul{(y^{'}+\gp^{r}\mcl{A}_{n,k}^{*})\cap(y+\gp^{r^{'}}\mcl{A}_{n,k}^{*})}}$.
\end{center}

Here we observe that the middle bijection extends to an automorphism as given below
\begin{center}
$\mattwo{I_{k}}{0_k}{2\gp^{r-r^{'}}I_{k}}{-I_{k}}$
\end{center}
of the finite torsion module
$\mcl{A}_{n-r,k}\times\mcl{A}_{n-r^{'},k} = \grpp$ for $\gl =
((n-r)^k,(n-r^{'})^k) \in \Gl$. This proves the equality of the
cardinality of sets in case $1$.

Now we note that the existence of a solution to the congruences
\begin{equation*}
\begin{aligned}
x \equiv & \ y  \ (mod\ \ \gp^r)\\
x \equiv & \ y' \ (mod\ \ \gp^{r'})\\
\end{aligned}
\end{equation*}
with conditions
\begin{equation*}
\begin{aligned}
x-y     \in & \  \gp^r \mcl{A}_{n,k}^{*}\\
x-y'    \in & \  \gp^{r'} \mcl{A}_{n,k}^{*}\\
\end{aligned}
\end{equation*}
is equivalent to existence of a solution to the congruences
\begin{equation*}
\begin{aligned}
x \equiv & \ y  \ (mod\ \ \gp^{r'})\\
x \equiv & \ y' \ (mod\ \ \gp^r)\\
\end{aligned}
\end{equation*}
with conditions
\begin{equation*}
\begin{aligned}
x-y     \in & \  \gp^{r'} \mcl{A}_{n,k}^{*}\\
x-y'    \in & \  \gp^r \mcl{A}_{n,k}^{*}.\\
\end{aligned}
\end{equation*}
To exactly find the cardinality of the number of solutions, we use standard filtrations of
$\AAA {n}$ and deduce that the number of solutions to both these sets of equations with the respective given conditions is\\
$\begin{cases}
        q^{(n-r-1)k}(q^k-1) &   \text{ if } r > r' \\
        q^{(n-r'-1)k}(q^k-1) &  \text{ if } r'> r \\
    q^{(n-r-1)k}(q^k-1) & \text{ if } r = r' \text{ and } y-y' \in \gp^{(r+1)}\mcl{A}_{n,k}\\
    q^{(n-r-1)k}(q^k-2) & \text{ if } r = r' \text{ and } y-y' \in \gp^{r}\mcl{A}_{n,k}^{*}
\end{cases}$

The other cases $(2)$ and $(3)$ are similar.
\end{proof}
\begin{defn}
Let $0 \neq x \in \mbb{F}_q^n$ and $S \subs \mbb{F}_q^n$. We say $S \subs \mbb{F}_q^n$ is
linearly independent to $x$, if $x \nin S$ and the set $\{s,x\}$ is linearly independent for all
$s \in S$. Let $x,y \in \mbb{F}_q^n$ with $x \neq y$ we say $x$ is linearly independent with $y$
again if the set $\{x,y\}$ is linearly independent.
\end{defn}
\begin{lemma}
\label{lemma:countingvectors} Let $x$ be a nonzero vector in $\mbb{F}_q^{k}$.
Let $\A$ be a discrete valuation ring with a uniformizing parameter $\gp$. Denote by $S
\subs \A_{n}^k$ the set consisting of $k$-tuples such that $S
(mod\ \ \gp)$ is a set of vectors in $\mbb{F}_q^{k}$ linearly independent to $x$.
Let $a,b \in \A_{n}^k$ be two elements such that $a \equiv b (mod\ \ \gp^r)$ where
$0\leq r < n$. Consider the equations
\begin{equation}
\begin{aligned}
\label{eq:Case8}
e \equiv & \ a \ (mod\ \ \gp^r)\\
e \equiv & \ b \ (mod\ \ \gp^r)
\end{aligned}
\end{equation}
with conditions
\begin{equation}
\begin{aligned}
\label{eq:Condition8}
\gp^{-r}(e - a) \in & \ S\\
\gp^{-r}(e - b) \in & \ S
\end{aligned}
\end{equation}
If there exists a solution to equations~(\ref{eq:Case8}) satisfying
conditions~(\ref{eq:Condition8}), then the total number of such
solutions is $
\begin{cases}
q^{k(n-r-1)}(q^{k}-2q) & \text{ if } \gp^{-r}(a-b) \in S\\
q^{k(n-r-1)}(q^{k}-q) & \text{ if } \gp^{-r}(a-b) \nin S
\end{cases}
$
\end{lemma}
\begin{proof}
Since $\AAA {} \cong_{\phi} \mbb{F}_q$. Let $t^k:\A^k \lra
\mbb{F}_q^k$ with $t^k = \phi^k\ o \ q^k$ where $q:\A \lra (\AAA
{})$ be the quotient map. Let $s^k:\mbb{F}_q^k\lra \A^k$ be any
section. Then given any element $c \in \A_{n}^k$ there exists a
unique set $\{c_0,c_1,c_2,...,c_{n-1}\}$ of vectors in $\mbb{F}_q^k$
such that
\begin{equation}
c = s^k(c_0) + s^k(c_1)\gp + s^k(c_2)\gp^2 + ...+ s^k(c_{n-1})\gp^{n-1}
\end{equation}
with condition
\begin{equation}
t^k(s^k(c_i)) = c_i \text{ for all } i = 0,1,2,...,(n-1)
\end{equation}
Let
\begin{equation}
\begin{aligned}
a =& s^k(a_0) + s^k(a_1)\gp + s^k(a_2)\gp^2 + ...+ s^k(a_{n-1})\gp^{n-1}\\
b =& s^k(b_0) + s^k(b_1)\gp + s^k(b_2)\gp^2 + ...+ s^k(b_{n-1})\gp^{n-1}
\end{aligned}
\end{equation}
Since a solution to equations~(\ref{eq:Case8}) satisfying
conditions~(\ref{eq:Condition8}) exists, we have
\begin{equation*}
\begin{aligned}
a \equiv & \ b (mod\ \ \gp^r) \text{ and hence }\\
a_i = & \ b_i \in \mbb{F}_q^k \text{ for all } i = 0,1,2,...,(r-1)
\end{aligned}
\end{equation*}
The conditions~(\ref{eq:Condition8}) implies that we need to count
the number of solutions $e \in \A_{n}^k$ such that the vectors
\begin{equation*}
\begin{aligned}
e_i = a_i = b_i & \text{ for } 0 \leq i \leq (r-1).\\
e_r-a_r,e_r-b_r & \text{ each of which is linearly independent with } x.\\
e_i & \text{ can be any element in } \mbb{F}_q^k \text{ for } r+1 \leq i < n.\\
\end{aligned}
\end{equation*}
Suppose $a_r-b_r \nin S$. Then $\{a_r + \gep x \mid \gep \in
\mbb{F}_q\} = \{b_r + \gep x \mid \gep \in \mbb{F}_q\}$. So the
number of such solutions $e \in \A_{n}^k$ in this case is
$q^{(n-r-1)k}(q^k-q)$. Suppose $a_r-b_r \in S$. Then $\{a_r + \gep x
\mid \gep \in \mbb{F}_q\} \cap \{b_r + \gep x \mid \gep \in
\mbb{F}_q\} = \es$. So the number of such solutions $e \in \A_{n}^k$
in this case is $q^{(n-r-1)k}(q^k-2q)$.
\end{proof}
\begin{lemma}
\label{lemma:LinearIndependences} Let $s,r_{1},r_{2}$ be nonnegative
integers such that $s \leq r_{1}, s \leq r_{2},s \leq t$. Let $a,b
\in \gp^{s}\A_t^{\gr*}$. Let $y_1,y_2$ be two units in $\A^*$.
Suppose the residue field $\mbb{F}_q \cong \AAA {1}$ has at least
three elements or the multiplicity $\gr$ is $> 2$. Then the number
of solutions for $e \in \gp^{s}\A_t^{\gr*}$ to the congruences
\begin{equation}
\label{eq:Case1}
\begin{aligned}
b \equiv  & \ e y_2 \ (mod\ \ \gp^{r_{2}})\\
e \equiv    &   \   a y_1   \   (mod\   \   \gp^{r_{1}})
\end{aligned}
\end{equation}
with conditions
\begin{equation}
\label{eq:Condition1}
\begin{aligned}
&\{\gp^{-r_{2}}(b-e y_2)\ (mod\ \ \gp),& \gp^{-s}e\ (mod\ \ \gp)\} \text{ are linearly independent in } \mbb{F}_q^{\gr}\\
&\{\gp^{-r_{1}}(e-a y_1)\ (mod\ \ \gp),& \gp^{-s}a\ (mod\ \ \gp)\} \text{ are linearly independent in } \mbb{F}_q^{\gr}\\
\end{aligned}
\end{equation}
is the same as the number of solutions for $e \in
\gp^{s}\A_t^{\gr*}$ to the congruences
\begin{equation}
\label{eq:Case2}
\begin{aligned}
b \equiv  & \ e y_1 \ (mod\ \ \gp^{r_{1}})\\
e \equiv    &   \   a y_2   \   (mod\   \   \gp^{r_{2}})
\end{aligned}
\end{equation}
with conditions
\begin{equation}
\label{eq:Condition2}
\begin{aligned}
&\{\gp^{-r_{1}}(b-e y_1)\ (mod\ \ \gp),& \gp^{-s}e\ (mod\ \ \gp)\} \text{ are linearly independent in } \mbb{F}_q^{\gr}\\
&\{\gp^{-r_{2}}(e-a y_2)\ (mod\ \ \gp),& \gp^{-s}a\ (mod\ \ \gp)\} \text{ are linearly independent in } \mbb{F}_q^{\gr}\\
\end{aligned}
\end{equation}
Also the number of solutions for $e \in \gp^{s}\A_t^{\gr*}$ to the
congruences
\begin{equation}
\label{eq:Case4}
\begin{aligned}
b \equiv  & \ e y_2 \ (mod\ \ \gp^{r_{2}})\\
e \equiv    &   \   a y_1   \   (mod\   \   \gp^{r_{1}})
\end{aligned}
\end{equation}
with conditions
\begin{equation}
\label{eq:Condition4}
\begin{aligned}
&\{\gp^{-r_{2}}(b-e y_2)\ (mod\ \ \gp),& \gp^{-s}e\ (mod\ \ \gp)\} \text{ are linearly independent in } \mbb{F}_q^{\gr}\\
\end{aligned}
\end{equation}
is the same as the number of solutions for $e \in
\gp^{s}\A_t^{\gr*}$ to the congruences
\begin{equation}
\label{eq:Case5}
\begin{aligned}
b \equiv  & \ e y_1 \ (mod\ \ \gp^{r_{1}})\\
e \equiv    &   \   a y_2   \   (mod\   \   \gp^{r_{2}})
\end{aligned}
\end{equation}
with conditions
\begin{equation}
\label{eq:Condition5}
\begin{aligned}
&\{\gp^{-r_{2}}(e-a y_2)\ (mod\ \ \gp),& \gp^{-s}a\ (mod\ \ \gp)\} \text{ are linearly independent in } \mbb{F}_q^{\gr}\\
\end{aligned}
\end{equation}
\end{lemma}
\begin{proof}
First let us look at the congruence equations~(\ref{eq:Case1}) with
conditions~(\ref{eq:Condition1}) and congruence
equations~(\ref{eq:Case2}) with conditions~(\ref{eq:Condition2}).
Without loss of generality, let $s\leq r_{1} \leq r_{2}$. Existence
of such a solution $\ul{e}$ in any of the equations implies
\begin{equation}
\label{eq:Case3}
\begin{aligned}
b \equiv a y_1 y_2\ (mod\ \ \gp^{r_{1}})\\
\end{aligned}
\end{equation}
and if $r_{1} < r_{2}$ then we also have
\begin{equation}
\label{eq:Condition3}
\begin{aligned}
&\{\gp^{-r_{1}}(b-a y_1 y_2)\ (mod\ \ \gp),& \gp^{-s}e\ (mod\ \ \gp)\} \text{ are linearly independent in } \mbb{F}_q^{\gr}\\
\end{aligned}
\end{equation}
If there exists an element $\ul{e}$ satisfying
equation~(\ref{eq:Case1}) and condition~(\ref{eq:Condition1}) then
we choose $\ul{\ti{e}} = a y_2 + \gp^{r_{2}} \ga \ (mod\ \ \gp^t)$
for some $\ga \in \A^{\gr}$ such that
\begin{itemize}
\item $\ga \ (mod\ \ \gp)$ is linearly independent with $\gp^{-s}a\ (mod\ \ \gp)$ in $\mbb{F}_q^{\gr}$.
\item $(\gp^{-r_{1}}(b-a y_1y_2)-\gp^{r_{2}-r_{1}}(y_1\ga) \ (mod\ \ \gp^t)) \ (mod\ \ \gp)$ is linearly independent with $\gp^{-s}a\ (mod\ \ \gp)$ in $\mbb{F}_q^{\gr}$
\end{itemize}
Existence of such an $\ga$ in the case when
\begin{itemize}
\item $r_{1}=r_{2}$
\item $\gp^{-r_{1}}(b-a y_1y_2)$ is linearly independent with $\gp^{-s}a\ (mod\ \ \gp)$ in $\mbb{F}_q^{\gr}$
\item The residue field $\AAA {} \cong \mbb{F}_q$ has exactly two elements
\end{itemize}
requires that the multiplicity $\gr$ must be $> 2$. This element
$\ul{\ti{e}}$ gives rise to a solution to equation~(\ref{eq:Case2})
satisfying the condition~(\ref{eq:Condition2}).

Conversely if there exists an element $\ul{e}$ satisfying
equation~(\ref{eq:Case2}) and condition~(\ref{eq:Condition2}) then
we choose $\ul{\ti{e}} = b y^{-1}_2 + \gp^{r_{2}} y^{-1}_2 \ga \
(mod\ \ \gp^t)$ for some $\ga \in \A^{\gr}$ such that
\begin{itemize}
\item $\ga \ (mod\ \ \gp)$ is linearly independent with $\gp^{-s}a\ (mod\ \ \gp)$ in $\mbb{F}_q^{\gr}$.
\item $(\gp^{-r_{1}}(b y^{-1}_2-ay_1)+\gp^{r_{2}-r_{1}}(y^{-1}_2\ga) \ (mod\ \ \gp^t)) \ (mod\ \ \gp)$ is linearly independent with $\gp^{-s}a\ (mod\ \ \gp)$ in $\mbb{F}_q^{\gr}$
\end{itemize}
Again existence of such an $\ga$ in the case when
\begin{itemize}
\item $r_{1}=r_{2}$
\item $\gp^{-r_{1}}(b y^{-1}_2-a y_1)$ is linearly independent with $\gp^{-s}a\ (mod\ \ \gp)$ in $\mbb{F}_q^{\gr}$
\item The residue field $\AAA {} \cong \mbb{F}_q$ has exactly two elements
\end{itemize}
requires that the multiplicity $\gr$ must be $> 2$.

This element $\ul{\ti{e}}$ gives rise to a solution to
equation~(\ref{eq:Case1}) satisfying the
condition~(\ref{eq:Condition1}).

To exactly count the cardinality of the number of solutions, we
use standard filtrations of $\A_{n}$ and deduce that the number of
solutions $e \in \A_t^{\gr}$ to the set of
equations~(\ref{eq:Case1}) satisfying the
conditions~(\ref{eq:Condition1}) is same as the number of solutions
$e \in \A_t^{\gr}$ to the set of equations~(\ref{eq:Case2})
satisfying the conditions~(\ref{eq:Condition2}) and it is given by
\begin{itemize}
\item $q^{\gr(t-r_2-1)}(q^{\gr}-q)$ if $s \leq r_{1} < r_{2}$\\
\item $|(ay_1y_2+\gp^{r_{1}=r_{2}}S)\cap(b+\gp^{r_{1}=r_{2}}S)|$ if $s < r_{1} = r_{2}$ where $S \subs \A_t^{\gr}$ is a set such that
$S \ (mod\ \ \gp)$ is a set of vectors each of which is linearly independent to $\gp^{-s}a\ (mod\ \ \gp)$ in $\mbb{F}_q^{\gr}$.
This cardinality can be easily calculated to be\\
$
\begin{cases}
q^{\gr(t-r-1)}(q^{\gr}-2q) \text{ if } (ay_1y_2-b) \in \gp^rS \text{ where } r=r_{1}=r_{2}.\\
q^{\gr(t-r-1)}(q^{\gr}-q) \text{ if } (ay_1y_2-b) \nin \gp^rS \text{
where } r=r_{1}=r_{2}.
\end{cases}
$
\item Cardinality of the set $\{e \in \gp^{s}\A_t^{\gr} \mid \gp^{-s}e\ (mod\ \ \gp)$ is linearly independent to both
$\gp^{-s}a\ (mod\ \ \gp)$  and $\gp^{-s}b\ (mod\ \ \gp)$ in
$\mbb{F}_q^{\gr}\}$ if $s = r_{1} = r_{2}$. Again this cardinality
can be easily calculated to be $q^{\gr(t-s-1)}(q^{\gr}-2q+1)$.
\end{itemize}

The proof of the cardinalities of the number of solutions is similar to the one given in Lemma~\ref{lemma:countingvectors}.

Now let us look at the congruence equations~(\ref{eq:Case4}) with conditions~(\ref{eq:Condition4}) and congruence
equations~(\ref{eq:Case5}) with conditions~(\ref{eq:Condition5}). Existence of such a solution $\ul{e}$ in any of these congruence equations implies
\begin{itemize}
\item If $r_{1} \leq r_{2}$ then
\begin{equation}
\label{eq:Case6}
\begin{aligned}
b \equiv a y_1 y_2\ (mod\ \ \gp^{r_{1}})\\
\end{aligned}
\end{equation}
\item If $r_{1} > r_{2}$ then
\begin{equation}
\label{eq:Case7}
\begin{aligned}
b \equiv a y_1 y_2\ (mod\ \ \gp^{r_{2}})\\
\end{aligned}
\end{equation}
and
\begin{equation}
\label{eq:Condition7}
\begin{aligned}
&\{\gp^{-r_{2}}(b-a y_1 y_2)\ (mod\ \ \gp),& \gp^{-s}a\ (mod\ \ \gp)\} \text{ are linearly independent in } \mbb{F}_q^{\gr}\\
\end{aligned}
\end{equation}
\end{itemize}
Suppose $r_{1} \leq r_{2}$. If there exists an element $\ul{e}$
satisfying equation~(\ref{eq:Case4}) and
condition~(\ref{eq:Condition4}) then we choose $\ul{\ti{e}} = a y_2
+ \gp^{r_{2}} \ga \ (mod\ \ \gp^t)$ for some $\ga \in \A^{\gr}$ such
that
\begin{itemize}
\item $\ga \ (mod\ \ \gp)$ is linearly independent with $\gp^{-s}a\ (mod\ \ \gp)$ in $\mbb{F}_q^{\gr}$.
\end{itemize}
This element $\ul{\ti{e}}$ gives rise to a solution to
equation~(\ref{eq:Case5}) satisfying the
condition~(\ref{eq:Condition5}).

Conversely if there exists an element $\ul{e}$ satisfying
equation~(\ref{eq:Case5}) and condition~(\ref{eq:Condition5}) then
we choose $\ul{\ti{e}} = b y^{-1}_2 + \gp^{r_{2}} y^{-1}_2 \ga \
(mod\ \ \gp^t)$ for some $\ga \in \A^{\gr}$ such that
\begin{itemize}
\item $\ga \ (mod\ \ \gp)$ is linearly independent with $\gp^{-s}a\ (mod\ \ \gp)$ in $\mbb{F}_q^{\gr}$.
\end{itemize}
This element $\ul{\ti{e}}$ gives rise to a solution to equation~(\ref{eq:Case4}) satisfying the condition~(\ref{eq:Condition4}).\\
Suppose $r_{1} > r_{2}$. If there exists an element $\ul{e}$
satisfying equation~(\ref{eq:Case4}) and
condition~(\ref{eq:Condition4}) then choose $\ul{\ti{e}} = b
y_1^{-1}$. This element $\ul{\ti{e}}$ gives rise to a solution to
equation~(\ref{eq:Case5}) satisfying the
condition~(\ref{eq:Condition5}).

Conversely if there exists an element $\ul{e}$ satisfying
equation~(\ref{eq:Case5}) and condition~(\ref{eq:Condition5}) then
we choose $\ul{\ti{e}} = a y_1$. This element $\ul{\ti{e}}$ gives
rise to a solution to equation~(\ref{eq:Case4}) satisfying the
condition~(\ref{eq:Condition4}).

To exactly count the cardinality of the number of solutions, we
use standard filtrations of $\A_{n}$ and deduce that the number of
solutions $e \in \A_t^{\gr}$ to the set of
equations~(\ref{eq:Case4}) satisfying the
conditions~(\ref{eq:Condition4}) is same as the number of solutions
$e \in \A_t^{\gr}$ to the set of equations~(\ref{eq:Case5})
satisfying the conditions~(\ref{eq:Condition5}) and it is given by
\begin{itemize}
\item $q^{\gr(t-r_{2}-1)}(q^{\gr}-q)$ if $s \leq r_{1} < r_{2}$\\
\item $q^{\gr(t-r_{1})}$ if $s \leq r_{2} < r_{1}$\\
\item $|(ay_1y_2+\gp^{r_{1}=r_{2}}\A_t^{\gr}\cap(b+\gp^{r_{1}=r_{2}}S)| = |(ay_1y_2+\gp^{r_{1}=r_{2}}S\cap(b+\gp^{r_{1}=r_{2}}\A_t^{\gr})|$
if $s < r_{1} = r_{2}$ where $S \subs \A_t^{\gr}$ is a set such that
$S \ (mod\ \ \gp)$ is a set of vectors linearly independent to
$\gp^{-s}a\ (mod\ \ \gp)$ in $\mbb{F}_q^{\gr}$. This cardinality can be easily calculated to be $q^{\gr(t-r-1)}(q^{\gr}-q)$ where $r=r_{1}=r_{2}$.\\
\item Cardinality of the set $\{e \in \gp^{s}\A_t^{\gr} \mid \gp^{-s}e\ (mod\ \ \gp)$ is linearly independent to $\gp^{-s}a\ (mod\ \ \gp)$ in
$\mbb{F}_q^{\gr}\}$ = Cardinality of the set $\{e \in
\gp^{s}\A_t^{\gr} \mid \gp^{-s}e\ (mod\ \ \gp)$ is linearly
independent to $\gp^{-s}b\ (mod\ \ \gp)$ in $\mbb{F}_q^{\gr}\}$ if
$s = r_{1} = r_{2}$ Again this cardinality can be easily calculated
and it is $q^{\gr(t-s-1)}(q^{\gr}-q)$.
\end{itemize}

Again the proof of the cardinalities of the number of solutions is similar to the one given in Lemma~\ref{lemma:countingvectors}.
\end{proof}
Now we can prove the commutativity of the convolution in the
endomorphism algebra $\Endo_{\autgp}(\mbb{C}[\grpp^{I^*}])$.
\begin{theorem}
\label{theorem:Commutativity} Let $\ugl$ be a partition. Let $\grpp$
be the corresponding finite $\A$-module. Let $\autgp$ be its
automorphism group. Let $I$ be an ideal in $J(\mcl{P})_{\gl}$.
Suppose $\mbb{F}_q \cong \AAA {}$ has at least three elements. Then
the endomorphism algebra $\Endo_{\autgp}(\mbb{C}[\grpp^{I^*}])$ is
commutative.
\end{theorem}
\begin{proof}
For $j=1,2$ let $\mcl{O}_j \subs \grpp^{I^*} \times \grpp^{I^*}$
denote two transitive sets. Let $(\ul{a},\ul{b}) \in \grpp^2$. Then
$\mcl{I}_{\mcl{O}_1}*\mcl{I}_{\mcl{O}_2}(\ul{a},\ul{b}) = 0 =
\mcl{I}_{\mcl{O}_2}*\mcl{I}_{\mcl{O}_1}(\ul{a},\ul{b})$ if
$(\ul{a},\ul{b}) \nin \grpp^{I^*} \times \grpp^{I^*}$. So assume
$(\ul{a},\ul{b}) \in \grpp^{I^*} \times \grpp^{I^*}$.

Now suppose $(\partial_{\gl_i}I,\gl_i) \in \max(I)$. Then the
$i^{th}$-component of orbit of pair corresponding to the
$i^{th}$-component $\gp^{\partial_{\gl_i}I}\A_{\gl_i}^{\gr_i*}$ of
the orbit $\grpp^{I^*}$ is given by
\begin{enumerate}[label=\Alph*$\bm{.}$]
\item
$(\mcl{O}_1)_{\gl_i} =  \{(a_i,b_i) \in
(\gp^{\partial_{\gl_i}I}\A_{\gl_i}^{\gr_i*} \times
\gp^{\partial_{\gl_i}I}\A_{\gl_i}^{\gr_i*}
\mid b_i-a_iy_1 \in \gp^{r_{i1}}\A_{\gl_i}^{\gr_i} \text{ for some } r_{i1} > \partial_{\gl_i}I \text{ and for some } y_1 \in \A^{*}\}$\\
$\textbf{ OR }$\\
\item
$(\mcl{O}_1)_{\gl_i}    =   \{(a_i,b_i) \in
\gp^{\partial_{\gl_i}I}\A_{\gl_i}^{\gr_i*} \times
\gp^{\partial_{\gl_i}I}\A_{\gl_i}^{\gr_i*} \mid b_i-a_iy_1 \in
\gp^{r_{i1}}\A_{\gl_i}^{\gr_i*}\text{ for some } r_{i1} \geq
\partial_{\gl_i}I \text{ and for some } y_1 \in \A^{*}
\text{ and } \gp^{-r_{i1}}(b_i-a_iy_1) (mod\ \ \gp)\\ \text{ is
linearly independent with }\gp^{-\partial_{\gl_i}I}a_i (mod\ \ \gp)
\text{ in } (\mbb{F}_q)^{\gr_i}\}$
\end{enumerate}

and

\begin{enumerate}[label=\alph*$\bm{.}$]
\item
$(\mcl{O}_2)_{\gl_i} =  \{(a_i,b_i) \in
\gp^{\partial_{\gl_i}I}\A_{\gl_i}^{\gr_i*}\times
\gp^{\partial_{\gl_i}I}\A_{\gl_i}^{\gr_i*}
\mid b_i-a_iy_2 \in \gp^{r_{i2}}\A_{\gl_i}^{\gr_i} \text{ for some } r_{i2} > \partial_{\gl_i}I \text{ and for some } y_2 \in \A^{*}\}$\\
$\textbf{ OR }$\\
\item
$(\mcl{O}_2)_{\gl_i}    =   \{(a_i,b_i) \in
\gp^{\partial_{\gl_i}I}\A_{\gl_i}^{\gr_i*} \times
\gp^{\partial_{\gl_i}I}\A_{\gl_i}^{\gr_i*} \mid b_i-a_iy_2 \in
\gp^{r_{i2}}\A_{\gl_i}^{\gr_i*} \text{ for some } r_{i2} \geq
\partial_{\gl_i}I \text{and for some } y_2 \in \A^{*}
\text{ and } \gp^{-r_{i2}}(b_i-a_iy_2) (mod\ \ \gp)\\ \text{ is
linearly independent with }\gp^{-\partial_{\gl_i}I}a_i (mod\ \ \gp)
\text{ in } (\mbb{F}_q)^{\gr_i}\}$
\end{enumerate}

Consider case $A$ and case $a$.  From the
Lemma~\ref{lemma:LemmaCommutativity} the number of solutions $e_i
\in \gp^{\partial_{\gl_i}I}\A_{\gl_i}^{\gr_i*}$ such that
$e_i-a_iy_1 \in \gp^{r_{i1}}\A_{\gl_i}^{\gr_i}$ and $b_i-e_iy_2 \in
\gp^{r_{i2}}\A_{\gl_i}^{\gr_i}$ is the same as the number of
solutions $e_i \in \gp^{\partial_{\gl_i}I}\A_{\gl_i}^{\gr_i*}$ such
that $e_i-a_iy_2 \in \gp^{r_{i2}}\A_{\gl_i}^{\gr_i}$ and $b_i-e_iy_1
\in \gp^{r_{i1}}\A_{\gl_i}^{\gr_i}$.

Consider case $B$ and case $a$. From the
Lemma~\ref{lemma:LinearIndependences} the number of solutions $e_i
\in \gp^{\partial_{\gl_i}I}\A_{\gl_i}^{\gr_i*}$ such that
$e_i-a_iy_1 \in \gp^{r_{i1}}\A_{\gl_i}^{\gr_i},
\gp^{-r_{i1}}(e_i-a_iy_1) (mod\ \ \gp) \text{ is linearly
independent with }\\ \gp^{-\partial_{\gl_i}I}a_i (mod\ \ \gp) \text{
in } (\mbb{F}_q)^{\gr_i}$ and $b_i-e_iy_2 \in
\gp^{r_{i2}}(\A_{\gl_i}^{\gr_i}$ is the same as the number of
solutions $e_i \in \gp^{\partial_{\gl_i}I}\A_{\gl_i}^{\gr_i*}$ such
that $e_i-a_iy_2 \in \gp^{r_{i2}}\A_{\gl_i}^{\gr_i}$ and
$b_i-e_iy_1 \in \gp^{r_{i1}}\A_{\gl_i}^{\gr_i},\\
\gp^{-r_{i1}}(b_i-e_iy_1) (mod\ \ \gp) \text{ is linearly
independent with } \gp^{-\partial_{\gl_i}I}e_i (mod\ \ \gp) \text{
in } (\mbb{F}_q)^{\gr_i}$.

Consider case $A$ and case $b$. This case is similar to the above case $B$ and case $a$.

Consider case $B$ and case $b$. Again from the
Lemma~\ref{lemma:LinearIndependences} the number of solutions $e_i
\in \gp^{\partial_{\gl_i}I}\A_ {\gl_i}^{\gr_i*}$ such that
$e_i-a_iy_1 \in \gp^{r_{i1}}\A_{\gl_i}^{\gr_i},
\gp^{-r_{i1}}(e_i-a_iy_1) (mod\ \ \gp) \text{ is linearly
independent with }\\ \gp^{-\partial_{\gl_i}I}a_i (mod\ \ \gp) \text{
in } (\mbb{F}_q)^{\gr_i}$ and $b_i-e_iy_2 \in
\gp^{r_{i2}}\A_{\gl_i}^{\gr_i}, \gp^{-r_{i2}}(b_i-e_iy_2) (mod\ \
\gp) \text{ is linearly }\\ \text{ independent with }
\gp^{-\partial_{\gl_i}I}e_i (mod\ \ \gp) \text{in }
(\mbb{F}_q)^{\gr_i}$ is the same as the number of solutions $e_i \in
\gp^{\partial_{\gl_i}I}\A_{\gl_i}^{\gr_i*}$ such that $e_i-a_iy_2
\in \gp^{r_{i2}}\A_{\gl_i}^{\gr_i}, \gp^{-r_{i2}}(e_i-a_iy_2) (mod\
\ \gp) \text{ is linearly independent with }\\
\gp^{-\partial_{\gl_i}I}a_i (mod\ \ \gp) \text{ in }
(\mbb{F}_q)^{\gr_i}$ and $b_i-e_iy_1 \in
\gp^{r_{i1}}\A_{\gl_i}^{\gr_i}, \gp^{-r_{i1}}(b_i-e_iy_1) (mod\ \
\gp)\\ \text{is linearly independent with }
\gp^{-\partial_{\gl_i}I}e_i (mod\ \ \gp) \text{ in }
(\mbb{F}_q)^{\gr_i}$.

Now suppose $(\partial_{\gl_i}I,\gl_i) \nin \max(I)$. Then the
$i^{th}$-component of orbit of pair corresponding to the
$i^{th}$-component $\gp^{\partial_{\gl_i}I}\A_{\gl_i}^{\gr_i}$ of
the orbit $\grpp^{I^*}$ is given by
\begin{enumerate}[label=\Alph*$\bm{.}$]
\item
$(\mcl{O}_1)_{\gl_i} =  \{(a_i,b_i) \in \gp^{\partial_{\gl_i}I}\A_{\gl_i}^{\gr_i} \times \gp^{\partial_{\gl_i}I}\A_{\gl_i}^{\gr_i}\}$\\
$\textbf{ OR }$\\
\item
$(\mcl{O}_1)_{\gl_i} =  \{(a_i,b_i) \in
\gp^{\partial_{\gl_i}I}\A_{\gl_i}^{\gr_i} \times
\gp^{\partial_{\gl_i}I}\A_{\gl_i}^{\gr_i} \mid
b_i-a_iy_1 \in \gp^{r_{i1}}\A_{\gl_i}^{\gr_i} \text{ for some } r_{i1} > \partial_{\gl_i}I \text{ and for some } y_1 \in \A^{*}\}$\\
$\textbf{ OR }$\\
\item
$(\mcl{O}_1)_{\gl_i}    =   \{(a_i,b_i) \in
(\gp^{\partial_{\gl_i}I}\A_{\gl_i}^{\gr_i} \times
\gp^{\partial_{\gl_i}I}\A_{\gl_i}^{\gr_i} \mid b_i-a_iy_1 \in
\gp^{r_{i1}}\A_{\gl_i}^{\gr_i*} \text{ for some } r_{i1} \geq
\partial_{\gl_i}I \text{ and for some } y_1 \in \A^{*}\}$
\end{enumerate}

and

\begin{enumerate}[label=\alph*$\bm{.}$]
\item
$(\mcl{O}_2)_{\gl_i} =  \{(a_i,b_i) \in \gp^{\partial_{\gl_i}I}\A_{\gl_i}^{\gr_i} \times \gp^{\partial_{\gl_i}I}\A_{\gl_i}^{\gr_i}\}$\\
$\textbf{ OR }$\\
\item
$(\mcl{O}_2)_{\gl_i} =  \{(a_i,b_i) \in
\gp^{\partial_{\gl_i}I}\A_{\gl_i}^{\gr_i} \times
\gp^{\partial_{\gl_i}I}\A_{\gl_i}^{\gr_i} \mid b_i-a_iy_2 \in
\gp^{r_{i2}}\A_{\gl_i}^{\gr_i} \text{ for some } r_{i2} >
\partial_{\gl_i}I \text{ and for some } y_2 \in
\A^{*}\}$
$\textbf{ OR }$\\
\item
$(\mcl{O}_1)_{\gl_i}    =   \{(a_i,b_i) \in
\gp^{\partial_{\gl_i}I}\A_{\gl_i}^{\gr_i} \times
\gp^{\partial_{\gl_i}I}\A_{\gl_i}^{\gr_i} \mid b_i-a_iy_2 \in
\gp^{r_{i2}}\A_{\gl_i}^{\gr_i*} \text{ for some } r_{i2} \geq
\partial_{\gl_i}I \text{ and for some } y_2 \in \A^{*}\}$
\end{enumerate}
Here in all pairs of the cases $\{A,B,C\} \times \{a,b,c\}$, we have from the Lemma~\ref{lemma:LemmaCommutativity},
the number of solutions $e_i$ in both way convolutions agree for each pair.

Hence $\mcl{I}_{\mcl{O}_1}*\mcl{I}_{\mcl{O}_2}(\ul{a},\ul{b}) =
\mcl{I}_{\mcl{O}_2}*\mcl{I}_{\mcl{O}_1}(\ul{a},\ul{b})$ for all
$(\ul{a},\ul{b}) \in \grpp^2$ and the endomorphism algebra
$\Endo_{\autgp}(\mbb{C}[\grpp^{I^*}])$ is commutative.
\end{proof}

Now we state the theorem regarding the multiplicity free nature of
the permutation representation.
\begin{theorem}
\label{theorem:MultFree} Let $\ugl$ be a partition. Let $\grpp$ be
the corresponding finite $\A$-module. Let $\autgp$ be its
automorphism group. Let $I$ be an ideal in $J(\mcl{P})_{\gl}$.
Suppose $\mbb{F}_q \cong \AAA {}$ has at least three elements. Then
the permutation representation $\mbb{C}[\grpp^{I^*}]$ of $\autgp$ is
multiplicity-free.
\end{theorem}
\begin{proof}
From Theorem~\ref{theorem:Commutativity}, we observe that the
endomorphism algebra $\Endo_{\autgp}(\mbb{C}[\grpp^{I^*}])$ is
commutative. So it follows that the permutation representation
$\mbb{C}[\grpp^{I^*}]$ of $\autgp$ is multiplicity-free.
\end{proof}
\section{Appendix}
\label{sec:Appendix}
Let $\grpp$ denote the abelian group corresponding to the partition $\ugl$. In this section we prove the existence of a certain combinatorial lattice structure
among certain $B-$valued subsets of $\grpp$ defined below in order to motivate Definition~\ref{def:Combinatorial}.
This is a bigger lattice (additively closed for supremum) when compared to the lattice of characteristic subgroups of abelian $p-$groups for $p>2$.

\begin{defn}[B-Val]
Let $B = \{0,1\} \subs \N \cup \{0\}$. A B-val $\ul a$ is an element of $\us{l \in \N}{\bc} \{0,1\}^{l}$.
\end{defn}
\begin{defn}[B-Set]
Fix an odd prime $p$ in this article. Let $\grpp$ denote the abelian group corresponding to
the partition $\ugl$. Let $(r_1,r_2,\ldots,r_k) \in \{1,2,\ldots,\gl_1\} \times \{1,2,\ldots,\gl_2\} \times \ldots \{1,2,\ldots,\gl_k\}$ and
$\ul{a}=(a_1,a_2,\ldots,a_k)$ $\in\nolinebreak \{0,1\}^{k}$. Define a $B-set$ denoted by $I(\ul r,\ul a)$ as the subset
$\us{i=1}{\os{k}{\prod}}p^{r_i}(\A^{\gr_i}_{\gl_i})^{a_i}$ of the abelian group $\grpp$ where $p^{r}(\A^{\gr}_{\gm})^{0}=p^{r}(\ZZ {\gm})^{\gr}$ and
$p^{r}(\A^{\gr}_{\gm})^{1}=p^{r}\big((\ZZ {\gm})^{\gr}\bs p(\ZZ {\gm})^{\gr}\big)$.
\end{defn}
\begin{defn}
\label{def:Lattice}
Let $\mcl{L}$ be a collection of subsets of $\grpp$. Define a lattice structure as follows. Let $L_1,L_2 \in \mcl{L}$ then $L_1 \wedge L_2$ is an element of $\mcl{L}$
which is the biggest set in $\mcl{L}$ which is contained in $L_1$ and $L_2$. Define $L_1 \vee L_2$ to be the smallest set in $\mcl{L}$ that contains
$L_1+L_2$. We note that $L_1+L_2=\{l_1+l_2 \in \grpp\mid l_1\in L_1 \subs \grpp,l_2 \in L_2 \subs \grpp\}$.
\end{defn}
\begin{remark}
In the case when the lattice $\mcl{L}$ is the lattice of characteristic subgroups this definition is not needed as any subgroup $L$ which contains two groups $L_1,L_2$ as subgroups also contains the group $L_1+L_2$ generated by both of them. i.e. Note the least upper bound(supremum) is
additively closed instead of just closed under the set union.

Also we note that the lattice of characteristic subgroups has two way distributivity $(L_1 \vee L_2) \wedge L_3=(L_1 \wedge L_3) \vee (L_2 \wedge L_3),(L_1 \wedge L_2) \vee L_3 =(L_1 \vee L_3) \wedge (L_2 \vee L_3)$. This is because of the following identities.
For any real numbers $r,s,t$ we have
\begin{equation*}
max(min(r,s),t)=min(max(r,t),max(s,t))
\end{equation*}
\begin{equation*}
min(max(r,s),t)= max(min(r,t),min(s,t))
\end{equation*}
\end{remark}

The following two lemmas prove that the set of all such $B-$valued subsets is a lattice with respect to the definition~\ref{def:Lattice}.
\begin{lemma}
\begin{center}
$I(\ul r,\ul a) \cap I(\ul s,\ul b) =
\left\{
\begin{array}{c}
I(\ul r \cup \ul s,\ul a + \ul b - \ul a\ul b) \\
\es\ \text{ if there exists } i \text{ such that } r_i<s_i, a_i=1\\ \text{ or } s_i<r_i, b_i=1.
\end{array}
\right.
$
\end{center}
where $\ul a \ul b=(a_1b_1,\ldots,a_kb_k)$ and $\ul r \cup \ul s=(max(r_i,s_i):i=1,2,\ldots,k)$.
\end{lemma}
\begin{proof}
Let $i \in \{1,2,3,\ldots,k\}$. If $r_i=s_i$ then
\begin{itemize}
\item  $p^{r_i}\A_{\gl_i}^{\gr_i*}\cap p^{s_i}\A_{\gl_i}^{\gr_i*}=p^{r_i=s_i}\A_{\gl_i}^{\gr_i*}$.
\item  $p^{r_i}\A_{\gl_i}^{\gr_i}\cap p^{s_i}\A_{\gl_i}^{\gr_i*}=p^{r_i=s_i}\A_{\gl_i}^{\gr_i*}$.
\item  $p^{r_i}\A_{\gl_i}^{\gr_i*}\cap p^{s_i}\A_{\gl_i}^{\gr_i}=p^{r_i=s_i}\A_{\gl_i}^{\gr_i*}$.
\item  $p^{r_i}\A_{\gl_i}^{\gr_i}\cap p^{s_i}\A_{\gl_i}^{\gr_i}=p^{r_i=s_i}\A_{\gl_i}^{\gr_i}$.
\end{itemize}
If $r_i \neq s_i$ then assume without loss of generality $r_i<r_i+1\leq s_i$.
\begin{itemize}
\item  $p^{r_i}\A_{\gl_i}^{\gr_i*}\cap p^{s_i}\A_{\gl_i}^{\gr_i*}=\es$.
\item  $p^{r_i}\A_{\gl_i}^{\gr_i}\cap p^{s_i}\A_{\gl_i}^{\gr_i*}=p^{s_i}\A_{\gl_i}^{\gr_i*}$.
\item  $p^{r_i}\A_{\gl_i}^{\gr_i*}\cap p^{s_i}\A_{\gl_i}^{\gr_i}=\es$.
\item  $p^{r_i}\A_{\gl_i}^{\gr_i}\cap p^{s_i}\A_{\gl_i}^{\gr_i}=p^{s_i}\A_{\gl_i}^{\gr_i}$.
\end{itemize}
This proves the lemma.
\end{proof}
\begin{lemma}[Sum of Two B-Sets Lemma]
Let $I(\ul r,a),I(\ul s,b)$ denote two B-subsets of $\grpp$ with B-val $\ul a,\ul b$. Then $I(\ul r,\ul a)+I(\ul s,\ul b)$ is a B-set. The B-val associated to the B-set $I(\ul r,\ul a)+I(\ul s,\ul b)$ is given by $\ul c = 1/2((a+b)sgn(r-s)-(a-b))sgn(r-s)$.
The B-set $I(\ul r,\ul a)+I(\ul s,\ul b) = I(\ul r \cap \ul s,\ul c)$ where $\ul r \cap \ul s=(min(r_i,s_i):i=1,2,\ldots,k)$.
\end{lemma}
\begin{proof}
We have that $p$ is an odd prime. Let $i \in \{1,2,3,\ldots,k\}$. If $r_i=s_i$ then
\begin{itemize}
\item  $p^{r_i}\A_{\gl_i}^{\gr_i*}+ p^{s_i}\A_{\gl_i}^{\gr_i*}=p^{r_i=s_i}\A_{\gl_i}^{\gr_i},p$ is an odd prime.
\item  $p^{r_i}\A_{\gl_i}^{\gr_i}+ p^{s_i}\A_{\gl_i}^{\gr_i*}=p^{r_i=s_i}\A_{\gl_i}^{\gr_i}$.
\item  $p^{r_i}\A_{\gl_i}^{\gr_i*}+ p^{s_i}\A_{\gl_i}^{\gr_i}=p^{r_i=s_i}\A_{\gl_i}^{\gr_i}$.
\item  $p^{r_i}\A_{\gl_i}^{\gr_i}+ p^{s_i}\A_{\gl_i}^{\gr_i}=p^{r_i=s_i}\A_{\gl_i}^{\gr_i}$.
\end{itemize}
If $r_i \neq s_i$ then assume without loss of generality $r_i<r_i+1\leq s_i$.
\begin{itemize}
\item  $p^{r_i}\A_{\gl_i}^{\gr_i*}+p^{s_i}\A_{\gl_i}^{\gr_i*}=p^{r_i}\A_{\gl_i}^{\gr_i*}$.
\item  $p^{r_i}\A_{\gl_i}^{\gr_i}+p^{s_i}\A_{\gl_i}^{\gr_i*}=p^{r_i}\A_{\gl_i}^{\gr_i}$.
\item  $p^{r_i}\A_{\gl_i}^{\gr_i*}+p^{s_i}\A_{\gl_i}^{\gr_i}=p^{r_i}\A_{\gl_i}^{\gr_i*}$.
\item  $p^{r_i}\A_{\gl_i}^{\gr_i}+p^{s_i}\A_{\gl_i}^{\gr_i}=p^{r_i}\A_{\gl_i}^{\gr_i}$.
\end{itemize}
Here $sgn(r_i-s_i)=+1,0,-1$ according as $r_i>s_i,r_i=s_i,r_i<s_i$ respectively and $sgn(r-s)=(sgn(r_i-s_i)\mid i=1,2,\ldots,k)$.
Now the sum set equality
\begin{equation*}
I(\ul{r},\ul{a})+I(\ul{s},\ul{b})=I(\ul{r}\cap \ul{s},\ul{c}).
\end{equation*}
follows.
\end{proof}
\section{Acknowledgements}
The author is supported by an Indian Statistical Institute (ISI)
Grant in the position of the Visiting Scientist at ISI Bangalore,
India and also by an  Institute of Mathematical Sciences (IMSc)
Grant as Senior Research Scholar at IMSc Chennai, India. He thanks
Prof. Amritanshu Prasad, Prof. S. Viswanath, Prof. Vijay Kodiyalam
of IMSc, Chennai, India and Prof. B Sury of ISI, Bangalore, India.
This is also a part of the Ph.D thesis of the author while doing
Phd. at IMSc. Chennai.

\bibliographystyle{abbrv}
\def\cprime{$'$}

\end{document}